\documentclass[twoside,letterpaper,11pt,leqno]{amsart}

\usepackage[plainpages=false,pdfpagelabels,pdfencoding=auto,psdextra,colorlinks=true,urlcolor=black,linkcolor=black,citecolor=blue,hidelinks]{hyperref}

\usepackage[tmargin=1in, bmargin=1in, lmargin=1.49in, rmargin=1.49in]{geometry}

\usepackage{lmodern}
\usepackage[utf8x]{inputenc}
\usepackage[T1]{fontenc}
\usepackage[full]{textcomp}

\usepackage{amsopn}
\usepackage{amsmath}
\usepackage{amsfonts}
\usepackage{amssymb}
\usepackage{graphicx}
\usepackage{epstopdf}
\usepackage{float}
\usepackage{mathrsfs}
\usepackage{mathtools}
\usepackage[inline]{enumitem}
\usepackage{bm}
\usepackage[only,llbracket,rrbracket]{stmaryrd}
\SetSymbolFont{stmry}{bold}{U}{stmry}{m}{n}
\usepackage[caption=false]{subfig}
\usepackage{underscore}
\usepackage{multirow}
\usepackage[capitalize,nameinlink]{cleveref}

\newcommand\where{; \allowbreak \nonscript\; \mathopen{}}

\DeclarePairedDelimiterX\set[1]{\lbrace}{\rbrace}{\nonscript\,  #1 \nonscript\,}

\newcommand\restr[2]{{\left.\kern-\nulldelimiterspace #1 \right\rvert_{#2}}}
\newcommand\restrr[2]{{\kern-\nulldelimiterspace #1 \rvert_{#2}}}

\newcommand{\ol}[1]{\overline{#1}}

\newcommand{\lp}{\left(}
\newcommand{\rp}{\right)}
\newcommand{\lb}{\left[}
\newcommand{\rb}{\right]}
\newcommand{\ldb}{\llbracket}
\newcommand{\rdb}{\rrbracket}

\newcommand{\li}{\langle}
\newcommand{\ri}{\rangle}
\newcommand{\lli}{\left\li}
\newcommand{\rri}{\right\ri}
\newcommand{\lv}{\lvert}
\newcommand{\rv}{\rvert}

\newcommand{\lV}{\lVert}
\newcommand{\rV}{\rVert}

\newcommand\bff{\bm{f}}

\newcommand\bq{\bm{q}}
\newcommand\br{\bm{r}}

\newcommand\bu{\bm{u}}
\newcommand\bv{\bm{v}}

\newcommand\bx{\bm{x}}

\newcommand\bzero{\bm{0}}

\newcommand\bbR{\mathbb{R}}

\newcommand\cA{\mathcal{A}}
\newcommand\cB{\mathcal{B}}

\newcommand\cE{\mathcal{E}}
\newcommand\cF{\mathcal{F}}

\newcommand\cI{\mathcal{I}}

\newcommand\cO{\mathcal{O}}
\newcommand\cP{\mathcal{P}}

\newcommand\cS{\mathcal{S}}
\newcommand\cT{\mathcal{T}}
\newcommand\cU{\mathcal{U}}

\newcommand\ff{\mathfrak{f}}

\newcommand\fD{\mathfrak{D}}

\newcommand\fT{\mathfrak{T}}

\newcommand\sM{\mathscr{M}}

\newcommand\hr{\hat{r}}

\newcommand\hv{\hat{v}}

\newcommand\bhr{\bm{\hr}}

\newcommand\bhv{\bm{\hv}}

\newcommand\tv{\tilde{v}}

\newcommand\btv{\bm{\tv}}

\newcommand\wtB{\widetilde{B}}

\newcommand\half{\frac{1}{2}}
\newcommand\halff{{1/2}}

\newcommand\mone{{-1}}

\newcommand{\nn}{\nonumber}
\newcommand\col{\colon}

\newcommand\dd{\mathop{}\!\mathrm{d}}

\newcommand\grad{\nabla}

\let\div\relax\DeclareMathOperator{\div}{div}
\DeclareMathOperator{\curl}{curl}

\DeclareMathOperator{\diag}{diag}

\DeclareMathOperator{\NNZ}{NNZ}

\allowdisplaybreaks[4]

\newtheorem{theorem}{Theorem}
\newtheorem{lemma}[theorem]{Lemma}

\newtheorem{corollary}[theorem]{Corollary}
\theoremstyle{remark}
\newtheorem{remark}[theorem]{Remark}

\numberwithin{theorem}{section}
\numberwithin{equation}{section}

\title[Preconditioning using Interior Penalty]{Auxiliary Space Preconditioning of Finite Element Equations Using a Nonconforming Interior Penalty Reformulation and Static Condensation}
\author{Delyan Z. Kalchev} 
\address{Center for Applied Scientific Computing, Lawrence Livermore National Laboratory, P.O. Box 808, L-561, Livermore, CA 94551, USA.}
\email{kalchev1@llnl.gov}
\author{Panayot S. Vassilevski} 
\address{Department of Mathematics and Statistics, Portland State University, Portland, OR 97207, USA, and Center for Applied Scientific Computing, Lawrence Livermore National Laboratory, P.O. Box 808, L-561, Livermore, CA 94551, USA.}
\email{panayot@pdx.edu, vassilevski1@llnl.gov}

\thanks{This work was performed under the auspices of the U.S. Department of Energy by Lawrence Livermore National Laboratory under Contract DE-AC52-07NA27344 (LLNL-JRNL-788660).}
\thanks{The work of the second author was partially supported by NSF under grant DMS-1619640.}

\newcommand\Element{\textsc{element}}
\newcommand\Face{\textsc{face}}
\newcommand\Entity{\textsc{entity}}
\newcommand\Elements{\textsc{elements}}
\newcommand\Faces{\textsc{faces}}
\newcommand\Entities{\textsc{entities}}
\newcommand\Dof{\textsc{d}of}
\newcommand\Dofs{\textsc{d}ofs}
\newcommand\Edof{\textsc{e}dof}
\newcommand\Edofs{\textsc{e}dofs}
\newcommand\Bdof{\textsc{b}dof}
\newcommand\Bdofs{\textsc{b}dofs}
\newcommand\Adof{\textsc{a}dof}
\newcommand\Adofs{\textsc{a}dofs}

\begin{document}

\begin{abstract}
We modify the well-known interior penalty finite element discretization method so that it allows for element-by-element assembly. This is possible due to the introduction of additional unknowns associated with the interfaces between neighboring elements.
The resulting bilinear form, and a Schur complement (reduced) version of it, are utilized in a number of auxiliary space preconditioners for the original conforming finite element discretization problem. These preconditioners are analyzed on the fine scale and their performance is illustrated on model second order scalar elliptic problems discretized with high order elements.

\end{abstract}

\maketitle

\section{Introduction}

The well-known interior penalty (IP) discretization method \cite{1982IP,1976IP,1978IP,2003IPnonmatch} couples degrees of freedom across two neighboring elements, so it does not possess the element-by-element assembly property, which is inherent to conforming finite element discretization methods. The element-by-element assembly property is useful, for example, in ``matrix-free'' computations, since it minimizes the coupling across element interfaces. Also, certain element-based coarsening as in the AMGe methods, \cite{VassilevskiMG},  can be employed then, such that it  maintains the element-by-element assembly property on coarse levels. 
The modification we propose consists of the following simple idea. The IP bilinear form originally contains a jump term $\int_\ff\ldb u \rdb \ldb v \rdb\dd\rho$, over each interface $\ff$ between any two neighboring elements $\tau_-$ and $\tau_+$. Here, this term is replaced by two other terms, involving a new unknown, $u_\ff$, associated with $\ff$. We have then
$\int_\ff (u_- - u_\ff)(v_- - v_\ff)\dd\rho + \int_\ff(u_+ - u_\ff)(v_+ - v_\ff)\dd\rho$, where $u_-$ and $v_-$ come from the element $\tau_-$, while $u_+$ and $v_+$ come from the other element, $\tau_+$. Clearly, letting
$u_\ff = \frac{1}{2}(u_- + u_+)$ (and similarly, $v_\ff = \frac{1}{2}(v_- + v_+)$), recovers the 
original jump term, scaled by $\frac{1}{2}$. It is also clear that introducing the additional space of functions $v_b = (v_\ff)$, associated with the set of interfaces $\set{\ff}$, allows associating each of the two new terms with a unique neighboring element, i.e.,
$\int_\ff (u_- - u_\ff)(v_- - v_\ff)\dd\rho$ with $\tau_-$ and
$\int_\ff(u_+ - u_\ff)(v_+ - v_\ff)\dd\rho$ with $\tau_+$. 
Consequently, the coupling occurs only through these interface unknowns. 
We consider $v_b = (v_\ff)$ discontinuous from face to face.

The introduction of interface unknowns is a simple and basic idea to decouple neighboring elements. For example, similar idea is utilized for hybridization (see, e.g., \cite{CiarletFEM,BoffiMFE}) of finite element methods. In the context of hybridization, the interface spaces are used for Lagrangian multipliers associated with constraints on jump conditions across elements, where those Lagrangian multipliers can also be interpreted as solution traces. Here, the approach is more direct and somewhat simpler -- the interface spaces are explicitly built and interpreted as trace spaces, the jump conditions are between the interfaces and the elements (on both sides of each interface) instead of between neighboring elements, and these conditions are introduced as part of the minimization functional rather than as constraints. Moreover, this paper is devoted to using the reformulations for building preconditioners rather than alternative discretization methods.

Clearly, the modification can be employed with $\tau_-$ and $\tau_+$ replaced by subdomains
$T_-$ and $T_+$ (for example, being unions of finite elements) and $\ff$ replaced by 
the interface $F$ between $T_-$ and $T_+$. This is the approach we exploit, when building preconditioners for an original conforming discretization (i.e., no interior penalty terms to begin with). One may view the set of subdomains, $T$, as a coarse triangulation
$\cT^H$. 
In that case, each $T \in \cT^H$ is a union of fine elements from an initial fine triangulation $\cT^h$. We refer to $T$ as an \emph{agglomerated element}, or simply an {\em agglomerate}.
The modified IP method employs two sets of discontinuous spaces: one space of functions $u_e = (u_T)$ associated with 
the agglomerates $T \in \cT^H$ and the space of functions $u_b= (u_F)$ associated with the interfaces $\set{F}$ between any two neighboring $T_-$ and $T_+$ from $\cT^H$. 
The resulting bilinear form consists of the local bilinear forms  $a_T(u_T,v_T) + \frac{1}{\delta} \sum_{F \subset \partial T} \int_F (u_T - u_F)(v_T-v_F)\dd\rho$ associated with each $T$ ($\delta > 0$ is the penalty parameter).  
The trial functions are $u_e = (u_T)$ and $u_b = (u_F)$, whereas, similarly, the test functions are  $v_e=(v_T)$ and $v_b = (v_F)$. Here, $a_T(\cdot,\cdot)$ is a local, on $T$, version of the bilinear form in the original conforming discretization.

The goal of the present paper is to study the modified IP bilinear form as a tool for building preconditioners for the original conforming bilinear form utilizing the auxiliary space technique which goes back to S. Nepomnyaschikh (see \cite{2007DD}) and studied in detail by J. Xu \cite{1996AuxiliarySpace}. 
We analyze both additive and multiplicative versions of the auxiliary space preconditioners 
for general smoothers, following \cite[Theorem 7.18]{VassilevskiMG},  by verifying the assumptions needed there. 
Another use of the element-by-element assembly property of the modified IP method is the application of the spectral AMGe technique to construct algebraic multigrid (AMG) preconditioners for the IP bilinear form and for its reduced Schur complement form. We note that the modified IP form allows for static condensation, i.e., we can eliminate the $u_T$ unknowns (they are decoupled from each other) and, thus, end up with a problem for the interface unknowns, $u_F$, only. 

We have implemented these auxiliary space preconditioners and tested their theoretically proven mesh independent performance on model 2D and 3D scalar second order elliptic problems, including high order elements.  

The remainder of the paper is structured as follows. \Cref{sec:basic} introduces basic concepts, spaces, notation, and a problem of interest. The IP formulation and the respective auxiliary space preconditioners are presented in \cref{sec:IP}. Their properties are studied, showing (\cref{thm:ipopt}) the general optimality of a fine-scale auxiliary space preconditioning strategy via the IP reformulation. \Cref{sec:smoother} describes a polynomial smoother employed in the preconditioner. A generic AMGe approach that can be utilized for solving the IP auxiliary space problem is outlined in \cref{sec:nonconfprecond}. Numerical examples are demonstrated in \cref{sec:numerical}, while conclusions and future work are left for the final \cref{sec:conclusion}.

\section{Basic setting}
\label{sec:basic}

This section is devoted to providing foundations by addressing generic notions and basic procedures. Notation and abbreviated designations are introduced to facilitate a more convenient presentation in the rest of the paper.

\subsection{Mesh and agglomeration}
\label{ssec:meshagg}

A domain $\Omega \subset \bbR^d$ ($d$ is the space dimension) with a Lipschitz-continuous boundary, a (fine-scale) triangulation $\cT^h = \set{\tau}$ of $\Omega$, and a finite element space $\cU^h$ on $\cT^h$ are given. The mesh $\cT^h$ is seen as a set of elements and respective associated faces, where a \emph{face} is the interface, of dimension $d-1$, between two adjacent elements. In general, the degrees of freedom (dofs) of $\cU^h$ are split into dofs associated with the interiors of the elements and dofs associated with the faces (cf. \cref{fig:nonconfsplitting}), where a face dof belongs to multiple elements and can also belong to multiple faces, whereas an interior dof always belongs to a single element. The focus of this paper is on $\cU^h$ consisting of continuous piece-wise polynomial functions, equipped with the usual nodal dofs.

Let $\cT^H = \set{T}$ be a partitioning of $\cT^h$ into non-overlapping connected sets of fine-scale elements, called \emph{agglomerate elements} or, simply, \emph{agglomerates}; see \cref{fig:aggs}. In general, $\cT^H$ can be obtained by partitioning the dual graph of $\cT^h$, which is a graph with nodes the elements in $\cT^h$, where two elements are connected by an edge if they share a face. It can be expressed as a relation element_element. That is, all $T \in \cT^H$ are described in terms of the elements $\tau \in \cT^h$ via a relation \Element_element. Note that capitalization indicates agglomerate entities in $\cT^H$, like {\Element} (short for ``agglomerate element''), \Face, or \Entity, whereas regular letters indicate fine-scale entities in $\cT^h$, like element, face, or entity. This convention, unless otherwise specified, is used in the rest of the paper. Using element_face next, provides
\[
\text{\Element_face} = \text{\Element_element} \times \text{element_face},
\]
where element_face represents the relation, in $\cT^h$, between elements and their adjacent faces. Then, an intersection procedure over the sets described by \Element_face constructs the agglomerate {\Faces} in $\cT^H$ as sets of fine-scale faces, expressed via a relation \Face_face, and related to {\Elements} in the form of \Element_\Face; cf. \cref{fig:faces}. Consequently, each {\Face} can be consistently recognized as the $(d-1)$-dimensional surface that serves as an interface between two adjacent {\Elements} in $\cT^H$. The set of obtained {\Faces} in $\cT^H$ is denoted by $\Phi^H = \set{F}$. Considering the dofs in $\cU^h$, the relations element_dof, face_dof, and
\begin{align*}
\text{\Element_dof} &= \text{\Element_element} \times \text{element_dof},\\
\text{\Face_dof} &= \text{\Face_face} \times \text{face_dof}
\end{align*}
are determined.

For additional information on relation tables (matrices) and their utilization, see \cite{VassilevskiMG,2002AMGeTopology}.

\begin{figure}
\centering
\subfloat[][{{\Elements} in 3D}]{\includegraphics[width=0.43\textwidth]{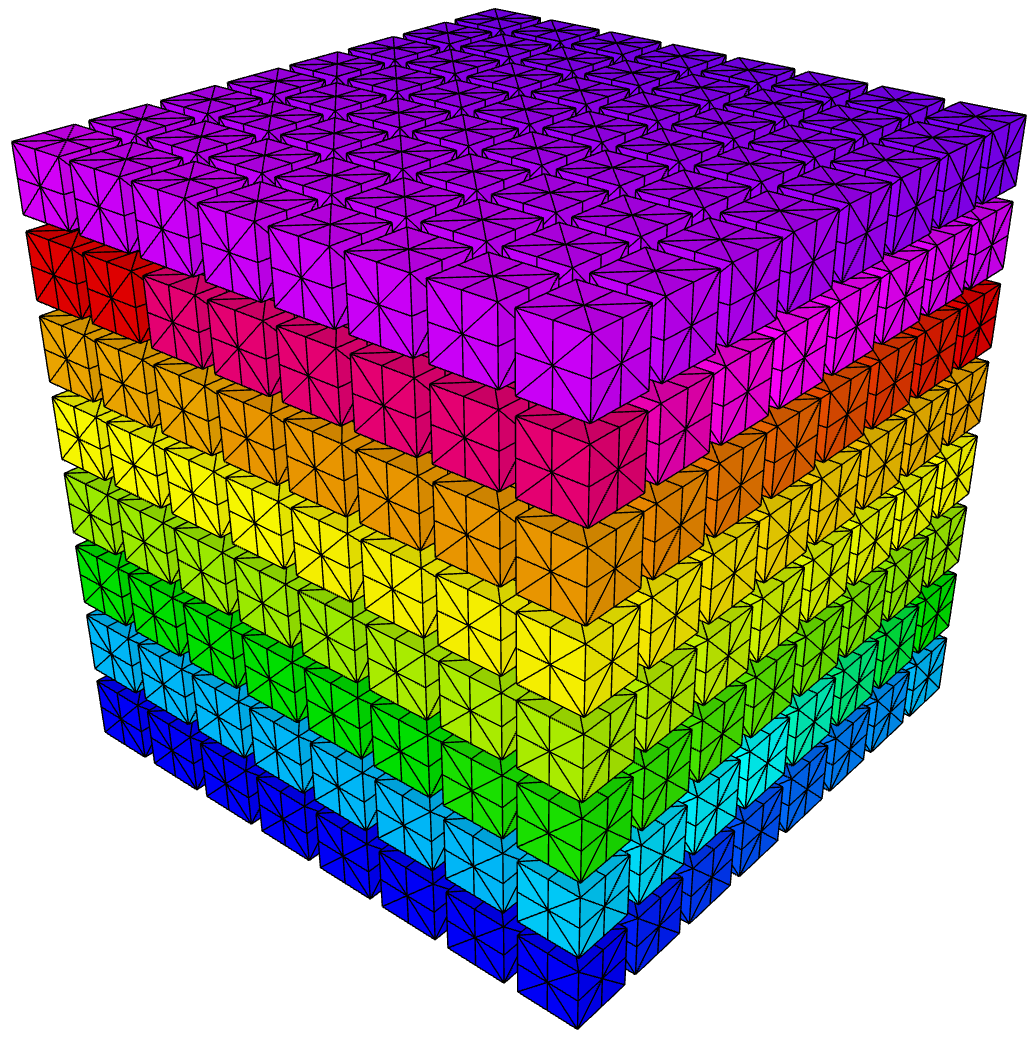}\label{fig:ELEMENTS}}\quad
\subfloat[][{{\Elements} in 2D}]{\includegraphics[width=0.35\textwidth]{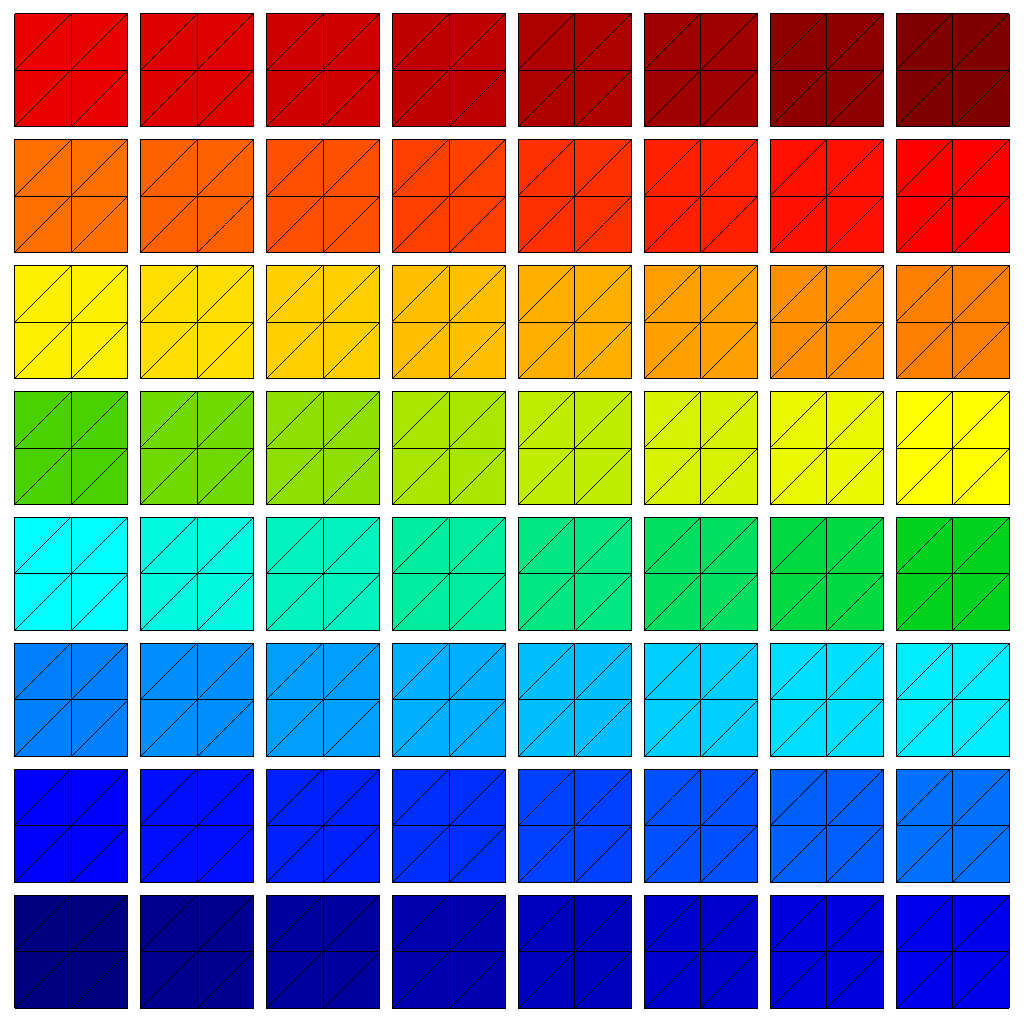}\label{fig:ELEMENTS2D}}
\caption[]{Examples of agglomerates (designated as {\Elements}) of (fine-scale) elements, utilized in the IP reformulation.}\label{fig:aggs}
\end{figure}

\begin{figure}
\centering
\includegraphics[width=0.50\textwidth]{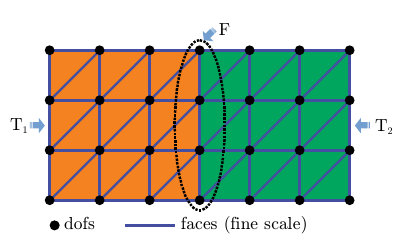}
\caption[]{An illustration of the designation of a {\Face} as a set of (fine-scale) faces, serving as an interface between {\Elements}. Here, the {\Elements} $T_1$ and $T_2$ are related to the {\Face} $F$ in the relation table \Element_\Face.}\label{fig:faces}
\end{figure}

\subsection{Pairs of nonconforming spaces}
\label{ssec:spaces}

\begin{figure}
\centering
\includegraphics[width=1.00\textwidth]{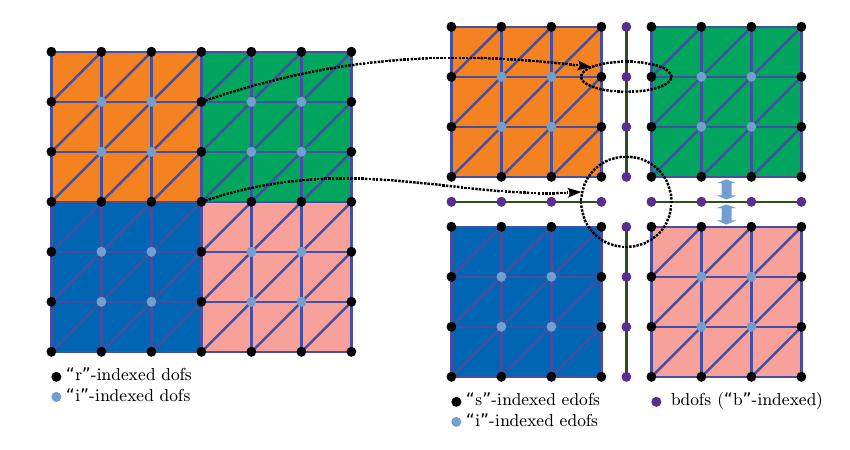}
\caption[]{An illustration of the construction of (fine-scale) nonconforming finite element spaces $\cE^h$ and $\cF^h$ from a conforming space $\cU^h$, utilizing agglomeration that provides {\Elements} and {\Faces}. Note that the only inter-{\Entity} coupling in the IP reformulation is between {\Elements} and their respective {\Faces}.}\label{fig:nonconfsplitting}
\end{figure}

A main idea in this work is to obtain discontinuous (nonconforming) finite element spaces and formulations on $\cT^H$ and $\Phi^H$. To that purpose, define the finite element spaces $\cE^h$ and $\cF^h$ via restrictions (or traces) of functions in $\cU^h$ onto $T \in \cT^H$ and $F \in \Phi^H$, respectively. Namely,
\begin{alignat*}{2}
\cE^h &= \Big\lbrace\, v_e^h \in L^\infty(\Omega) \where\;\; &&\forall\, T\in \cT^H,\; \exists\, v^h \in \cU^h\col\;\, \restr{v_e^h}{T} = \restr{v^h}{T} \,\Big\rbrace,\\
\cF^h &= \Big\lbrace\,v_b^h \in L^\infty(\cup_{F\in\Phi^H}F) \where\;\; &&\forall\, F\in \Phi^H,\; \exists\, v^h \in \cU^h\col\;\, \restr{v_b^h}{F} = \restr{v^h}{F} \,\Big\rbrace.
\end{alignat*}
Note that $\cE^h$ and $\cF^h$ are fine-scale spaces despite the utilization of agglomerate mesh structures like $\cT^H$ and $\Phi^H$, which justifies the parameter ``$h$'' (versus ``$H$''). Accordingly, the bases in $\cE^h$ and $\cF^h$ are derived via respective restrictions (or traces) of the basis in $\cU^h$. The degrees of freedom in $\cE^h$ and $\cF^h$ are obtained in a corresponding manner from the dofs in $\cU^h$, as illustrated in \cref{fig:nonconfsplitting}. For simplicity, ``dofs'' is reserved for the degrees of freedom in $\cU^h$, whereas ``edofs'' and ``bdofs'' are reserved for $\cE^h$ and $\cF^h$, respectively. Moreover, ``adofs'' designates the edofs and bdofs collectively and is associated with the product space $\cE^h\times\cF^h$. In more detail, edofs and bdofs are obtained by ``cloning'' all respective dofs for every agglomerate {\Entity} that contains the dofs, in accordance with \Element_dof and \Face_dof. For example, in \cref{fig:nonconfsplitting}: the dof in the center is cloned into four edofs (belonging to four separate {\Elements}) and four bdofs (belonging to four separate {\Faces}), thus, cloning that dof into eight adofs total; the dof in the interior of the {\Face} is cloned into three adofs -- one bdof (belonging to the {\Face} itself) and two edofs (belonging to two separate {\Elements}); and the dofs in the interiors of {\Elements} are simply copied as single edofs. Hence, each {\Entity} receives and it is the sole owner of a copy of all dofs it contains and there is no intersection between {\Entities} in terms of edofs and bdofs, i.e., they are completely separated without any sharing, making $\cE^h$ and $\cF^h$ spaces of discontinuous functions. Nevertheless, dofs, edofs, and bdofs are connected via their common ``ancestry'' founded on the above ``cloning'' procedure. Thus, the restrictions (or traces) of finite element functions in one of the spaces and their representations as functions in some of the other spaces is seen and performed in purely algebraic context. For example, if $\bq$ is a vector in terms of the edofs of some $T \in \cT^H$, that represents locally, on $T$, a function in $\cE^h$, and $F \subset \partial T$, then $\restr{\bq}{F}$ denotes a vector in terms of the bdofs of $F \in \Phi^H$, that represents locally, on $F$, a function in $\cF^h$. This is unambiguous and should lead to no confusion as it only involves a subvector and appropriate index mapping from edofs to bdofs, without any actual transformations, since $\cE^h$ and $\cF^h$ are trace (restriction) spaces for $\cU^h$. This constitutes an ``algebraic'' procedure, mimicking the behavior of restrictions (or traces) of finite element functions, on $T$, that vectors like $\bq$ represent, which are, generally, also supported on $F$. In what follows, finite element functions are identified with vectors on the degrees of freedom in the respective space.

Coarse subspaces $\cE^H \subset \cE^h$ and $\cF^H \subset \cF^h$ can be constructed by, respectively, selecting linearly independent vectors $\set{\bq_{T,i}}_{i=1}^{m_T}$, for every $T \in \cT^H$, and $\set{\bq_{F,i}}_{i=1}^{m_F}$, for every $F \in \Phi^H$, forming the bases for the coarse spaces. This is an ``algebraic'' procedure, formulating the coarse basis functions as linear combinations of fine-level basis functions, i.e., as vectors in terms of the fine-level degrees of freedom. The basis vectors are organized  appropriately as columns of \emph{prolongation} (or \emph{interpolation}) matrices $\cP_e\col \cE^H \mapsto \cE^h$ and $\cP_b\col \cF^H \mapsto \cF^h$, forming $\cP\col \cE^H \times \cF^H \mapsto \cE^h \times \cF^h$ as $\cP = \diag(\cP_e, \cP_b)$. For consistency, the corresponding degrees of freedom (associated with the respective coarse basis vectors) in $\cE^H$ and $\cF^H$ are respectively called ``\Edofs'' and ``\Bdofs'', and the collective term ``\Adofs'' is associated with $\cE^H\times\cF^H$. In essence, this is based on the ideas in AMGe (element-based algebraic multigrid) methods \cite{VassilevskiMG,2002AMGeTopology,2003AMGe,2007AMGe,2008AMGe,2012SAAMGe}. A particular way to obtain coarse basis vectors and construct coarse spaces, $\cE^H$ and $\cF^H$, is described in \cref{ssec:ipcoarseform}.

\subsection{Model problem}

The model problem considered in this paper is the second order scalar elliptic partial differential equation (PDE)
\begin{equation}\label{eq:pde}
-\div(\kappa\grad u) = f(\bx) \text{ in } \Omega,
\end{equation}
where $\kappa \in L^\infty(\Omega)$, $\kappa > 0$, is a given permeability field, $f \in L^2(\Omega)$ is a given source, and $u \in H^1(\Omega)$ is the unknown function. For simplicity of exposition, the boundary condition $u = 0$ on $\partial \Omega$, the boundary of $\Omega$, is considered, i.e., $u \in H^1_0(\Omega)$. The ubiquitous variational formulation $\min_{v \in H^1_0(\Omega)} [(\kappa\grad v, \grad v) - 2(f, v)]$ of \eqref{eq:pde} is utilized, providing the weak form
\begin{equation}\label{eq:weakform}
\text{Find } u \in H^1_0(\Omega)\col \; (\kappa \grad u, \grad v) = (f, v), \quad\forall v \in H^1_0(\Omega),
\end{equation}
where $(\cdot,\cdot)$ denotes the inner products in both $L^2(\Omega)$ and $[L^2(\Omega)]^d$. Consider the fine-scale finite element space $\cU^h \subset H^1_0(\Omega)$ defined on $\cT^h$. Using the finite element basis in $\cU^h$, \eqref{eq:weakform} induces the following linear system of algebraic equations:
\begin{equation}\label{eq:H1linsys}
A \bu = \bff,
\end{equation}
for the global symmetric positive definite (SPD) stiffness matrix $A$. Moreover, the local, on agglomerates, symmetric positive semidefinite (SPSD) stiffness matrices $A_T$, for $T \in \cT^H$, are obtainable, such that $A = \sum_{T \in \cT^H}A_T$ (the summation involves an implicit local-to-global mapping).

\section{Interior penalty approach}
\label{sec:IP}

In this section, an approach for obtaining preconditioners, based on the ideas of an \emph{interior penalty} (IP) method (see \cite{2003IPnonmatch}) is described and studied. A main idea is that, instead of using penalty terms on the jumps directly between {\Elements}, the interface space $\cF^h$ is employed to avoid direct coupling between {\Elements}.

\subsection{Formulation}
\label{ssec:ipformulation}

Consider the nonconforming discrete quadratic minimization formulation of \eqref{eq:pde}
\[
\min\sum_{T\in \cT^H}\lb \lp \kappa \grad v_e^h,\; \grad v_e^h \rp_T + \frac{1}{\delta} \sum_{F \subset \partial T} \lli \restr{v_e^h}{F} - v_b^h,\; \restr{v_e^h}{F} - v_b^h \rri_F \rb - 2(f, v_e^h),
\]
for $[v_e^h, v_b^h]\in \cE^h\times \cF^h$, providing, in lieu of \eqref{eq:weakform}, the weak form: Find $[u_e^h, u_b^h] \in \cE^h\times \cF^h$, such that
\begin{equation}\label{eq:ipweakform}
\sum_{T\in \cT^H}\lb \lp \kappa \grad u_e^h,\; \grad v_e^h \rp_T + \frac{1}{\delta} \sum_{F \subset \partial T} \lli \restr{u_e^h}{F} - u_b^h,\; \restr{v_e^h}{F} - v_b^h \rri_F \rb = (f, v_e^h),
\end{equation}
for all $[v_e^h, v_b^h] \in \cE^h\times \cF^h$. Here, $(\cdot,\cdot)_T$ and $\li\cdot,\cdot\ri_F$ denote the inner products in $[L^2(T)]^d$ and $L^2(F)$, respectively, and $\delta \in (0,1)$ is a small penalty parameter that can, generally, depend on $h$. The notation in \eqref{eq:ipweakform} requires an explanation. Namely, \eqref{eq:ipweakform} involves an implicit restriction of $u_e^h$ and $v_e^h$, which are generally discontinuous across $\partial T$, onto $T$. This removes the ambiguity from $u_e^h|_F$ and $v_e^h|_F$ by considering ``one-sided'' traces, involving the implicit restrictions to $T$ prior to further restricting to $F$. Therefore, the only inter-{\Entity} coupling, introduced by the bilinear form in \eqref{eq:ipweakform}, is between the edofs in $T \in \cT^H$ and the bdofs on $\partial T$. Thus, the edofs are not directly coupled across the {\Elements} and an assembly procedure for the matrix, corresponding to the bilinear form in \eqref{eq:ipweakform}, would only need to add (accumulate) contributions associated with bdofs.

Equation \eqref{eq:ipweakform} is slightly modified to obtain a formulation that is easier to implement algebraically and to facilitate the results in \cref{ssec:analysis} below. Particularly, only the interface penalty term is changed -- instead of the more standard $L^2(F)$ inner product in \eqref{eq:ipweakform}, $\li\cdot,\cdot\ri_F$, the inner product induced by the (restricted) diagonal of $A$ is utilized for the interface term. Consider $\cA_T$ -- the local, on $T \in \cT^H$, and corresponding to a summand in \eqref{eq:ipweakform}, matrix associated with the quadratic form
\begin{equation}\label{eq:localip}
\begin{bmatrix}
\bv_e\\
\bv_b
\end{bmatrix}^T
\cA_T
\begin{bmatrix}
\bv_e\\
\bv_b
\end{bmatrix}
= \bv_e^T A_T \bv_e + \frac{1}{\delta} \sum_{F\subset \partial T} \lp \restr{\bv_e}{F} - \restr{\bv_b}{F} \rp^T D_F \lp \restr{\bv_e}{F} - \restr{\bv_b}{F} \rp,
\end{equation}
where $\bv_e$ is defined on the edofs of $T$, $\bv_b$ -- on the bdofs on $\partial T$, $D_F$ is the restriction of the diagonal, $D$, of the global $A$ onto the bdofs of $F$, and $\cA_T$ takes the block form
\begin{equation}\label{eq:cAT}
\cA_T =
\begin{bmatrix}
\cA_{T,ee} &\cA_{T,eb}\\
\cA_{T,be} &\cA_{T,bb}
\end{bmatrix}.
\end{equation}
Here, the discussion in \cref{ssec:spaces}, concerning index mapping as part of the restriction, applies. Respective consistent, with the construction of $\cA_T$, relations {\Element}_edof, {\Element}_bdof, and, consequently, {\Element}_adof are also available. The global (SPD) IP matrix is obtainable via standard accumulation (assembly) -- $\cA = \sum_{T \in \cT^H}\cA_T$ -- which, as mentioned above, would involve only accumulation (addition) of the $\cA_{T,bb}$ blocks, while the remaining portions of the local matrices are simply copied into the global matrix.

Both $A_T$ and $\cA_T$ are SPSD with null spaces spanned by respective constant vectors (excluding essential boundary conditions). However, both $\cA_{T,bb}$ and $\cA_{T,ee}$ are positive definite due to the interface terms. In general, using that $A$ has strictly positive diagonal entries, $\cA_{T,ee}$ is positive definite if and only if every nonzero vector in the null space of $A_T$ has a nonzero value on some {\Face} $F \subset \partial T$. This is the case here, since the null space of $A_T$ is spanned by the constant vector.

\subsection{Auxiliary space preconditioners}
\label{ssec:ipaux}

Here, preconditioners based on the IP formulation are derived. Consider the finite element space $\cE^h \times \cF^h$, which in the context here is regarded as an ``auxiliary'' space. Let $\Pi_h\col \cE^h \times \cF^h \mapsto \cU^h$ be a linear transfer operator, to be defined momentarily. Identifying finite element functions with vectors allows to consider $\Pi_h$ as a matrix and obtain $\Pi_h^T\col \cU^h \mapsto \cE^h \times \cF^h$. In order to define the action of $\Pi_h$, recall that the relation between dofs, on one side, and adofs, on the other, is known and unambiguous. Therefore, it is reasonable to define the action of $\Pi_h$ as taking the arithmetic average, formulated in terms of adofs that correspond to a particular dof, of the entries of a given (auxiliary space) vector and obtaining the respective entries of a mapped vector defined on dofs. That is, for any dof $l$, let $J_l$ be the set of corresponding adofs (the respective ``cloned'' degrees of freedom) and consider a vector $\bhv$, defined in terms of adofs. Then,
\[
(\Pi_h \bhv)_l = \frac{1}{\lv J_l \rv} \sum_{j\in J_l} (\bhv)_j.
\]
Clearly, all row sums of $\Pi_h$ equal 1 and each column of $\Pi_h$ has exactly one nonzero entry. Assuming that $\cT^h$ is a regular (non-degenerate) mesh \cite{BrennerFEM}, $\lv J_l \rv$ is bounded, independently of $h$. That is,
\begin{equation}\label{eq:meshreg}
1 \le \lv J_l \rv \le \varkappa,
\end{equation}
for a constant $1 \le \varkappa < \infty$, which depends only on the regularity of $\cT^h$, but not on $h$. Indeed, let $\varkappa$ be the global maximum number of {\Elements} and {\Faces}, in $\cT^H$, that a dof can belong to, which, in turn, is bounded by the global maximum number of elements and faces, in $\cT^h$, that a dof can belong to.

Let $M$ be a ``smoother'' for $A$, such that $M^T + M - A$ is SPD, and $\cB$ -- a symmetric preconditioner for $\cA$. Define the \emph{additive auxiliary space preconditioner} for $A$
\begin{equation}\label{eq:Badd}
B^\mone_{\mathrm{add}} = \ol{M}^\mone + \Pi_h \cB^\mone \Pi_h^T,
\end{equation}
and the \emph{multiplicative auxiliary space preconditioner} for $A$
\begin{equation}\label{eq:Bmult}
B^\mone_{\mathrm{mult}} = \ol{M}^\mone + (I - M^{-T}A)\Pi_h \cB^\mone \Pi_h^T (I - A M^\mone),
\end{equation}
where $\ol{M} = M (M + M^T -A)^\mone M^T$ is the symmetrized (in fact, SPD) version of $M$. In case $M$ is symmetric, $\ol{M}^\mone$ in $B^\mone_{\mathrm{add}}$ can be replaced by $M^\mone$. The action of $B^\mone_{\mathrm{mult}}$ is obtained via a standard ``two-level'' procedure \cite{VassilevskiMG}:

Given $\bv_0 \in \bbR^{\dim(\cU^h)}$, $\bv_\mathrm{m} = B^\mone_{\mathrm{mult}} \bv_0$ is computed by the following steps:
\begin{enumerate}[label=(\roman*)]
\item ``pre-smoothing'': $\bv_1 = M^{-1} \bv_0$;
\item residual transfer to the auxiliary space: $\bhr = \Pi_h^T(\bv_0 - A \bv_1)$;
\item auxiliary space correction: $\bhv = \cB^\mone \bhr$;
\item correction transfer from the auxiliary space: $\bv_2 = \bv_1 + \Pi_h \bhv$;
\item ``post-smoothing'': $\bv_\mathrm{m} = \bv_2 + M^{-T}(\bv_0 - A \bv_2)$.
\end{enumerate}

\subsection{Analysis}\label{ssec:analysis}

Properties of the preconditioners, showing their optimality, are studied next. Define the operator $\cI_h\col  \cU^h \mapsto \cE^h \times \cF^h$, for\footnote{Finite element functions and algebraic vectors are identified, which should cause no confusion.} $\bv \in \cU^h$, via
\[
(\cI_h \bv)_j = (\bv)_l,
\]
for each adof $j$, where $j \in J_l$, for the corresponding dof $l$. This describes a procedure that appropriately copies the entries of $\bv$, so that the respective finite element functions, corresponding to $\bv$ and $\cI_h \bv$, can be viewed as coinciding in $H^1(\Omega)$. That is, in a sense, $\cI_h$ is an injection (embedding) of $\cU^h$ into $\cE^h \times \cF^h$. Considering the respective matrices, $\cI_h$ has the fill-in pattern of $\Pi_h^T$, but all nonzero entries are replaced by 1.

It is easy to see that $\Pi_h \cI_h = I$, the identity operator on $\cU^h$, implying that $\Pi_h$ is surjective, i.e., it has a full row rank. Moreover, for any $\bv \in \cU^h$, $\cI_h \bv \in \cE^h \times \cF^h$ (exactly) approximates $\bv$ in the sense
\begin{equation}\label{eq:approx}
\bv - \Pi_h \cI_h \bv = \bzero,
\end{equation}
and $\cI_h \bv$ is ``energy'' stable, since
\begin{equation}\label{eq:stable}
(\cI_h \bv)^T \cA\, \cI_h \bv = \bv^T A \bv.
\end{equation}
This is to be expected, since, in a sense, $\cE^h \times \cF^h$ ``includes'' $\cU^h$.

The more challenging task is to show the continuity of $\Pi_h\col \cE^h \times \cF^h \mapsto \cU^h$ in terms of the respective ``energy'' norms. This is addressed next.

\begin{lemma}\label{lem:approx}
The operator $E_h\col \cE^h \times \cF^h \mapsto \cE^h \times \cF^h$, defined as $E_h = \cI_h\Pi_h - I$, where $I$ is the identity on $\cE^h \times \cF^h$, is bounded in the sense
\[
(\cI_h\Pi_h \bhv - \bhv)^T \cA (\cI_h\Pi_h \bhv - \bhv) \le (1 + \Lambda\delta\varkappa^2)\, \bhv^T \cA \bhv,
\]
for all $\bhv \in \cE^h \times \cF^h$, where $\delta$ is the one in \eqref{eq:localip}, $\varkappa$ is from \eqref{eq:meshreg}, and $\Lambda > 0$ is a constant, independent of $h$, $H$, the coefficient $\kappa$ in \eqref{eq:pde}, and the regularity of $\cT^h$ (see \cref{rem:const}).
\end{lemma}
\begin{proof}
The portions of auxiliary space vectors corresponding to edofs and bdofs are respectively indexed by ``$e$'' and ``$b$'', leading to the notation, for $\bhv \in \cE^h \times \cF^h$, $\bhv^T = [\bhv_e^T, \bhv_b^T]^T$, where $\bhv_e \in \cE^h$ and $\bhv_b \in \cF^h$. By further splitting $\bhv_e^T = [\bhv_i^T, \bhv_s^T]^T$, it is obtained $\bhv^T = [\bhv_e^T, \bhv_b^T]^T = [\bhv_i^T, \bhv_s^T, \bhv_b^T]^T$, where ``$i$'' denotes the edofs in the interiors of all $T \in \cT^H$ and ``$s$'' are the edofs that can be mapped to some bdofs, based on the previously-described procedure of ``cloning'' dofs into edofs and bdofs. Analogously, the splitting $\bv^T = [\bv_i^T, \bv_r^T]^T$, for $\bv \in \cU^h$, is introduced in terms of dofs, indexed ``$i$'', in the interiors of all $T \in \cT^H$ and dofs, indexed ``$r$'', related to bdofs. \Cref{fig:nonconfsplitting} illustrates the utilized indexing. Note that there is a clear difference between ``$r$'' and ``$s$'', but also a clear relation. Locally, on $T \in \cT^H$, ``$r$'' and ``$s$'' can, in fact, be equated, but in a global setting, it is necessary to distinguish between ``$r$'' and ``$s$'' indices. This should not cause any ambiguity below. Based on the splittings, define
\begin{align}
A &=
\begin{bmatrix}
A_{ii} &A_{ir}\\
A_{ri} &A_{rr}
\end{bmatrix},\nn\\
\cA &=
\begin{bmatrix}
\cA_{ee} &\cA_{eb}\\
\cA_{be} &\cA_{bb}
\end{bmatrix}
=
\begin{bmatrix}
\cA_{ii} &\cA_{is}\\
\cA_{si} &\cA_{ss} &\cA_{sb}\\
&\cA_{bs} &\cA_{bb}
\end{bmatrix},\label{eq:cA}\\
\Pi_h &=
\begin{bmatrix}
\Pi_{ii}\\
&\Pi_{rs} &\Pi_{rb}
\end{bmatrix},\nn\\
\cI_h &=
\begin{bmatrix}
I\\
&\cI_{sr}\\
&\cI_{br}
\end{bmatrix},\nn
\end{align}
where $\cI_{sr}$ is the map from ``$r$'' to ``$s$'' indices and $\cI_{br}$ -- from ``$r$'' to ``$b$'' indices. Note that $\cI_{sr}$ is a matrix with the fill-in structure of $\Pi_{rs}^T$, where all nonzero entries are replaced by 1, each row has exactly one nonzero entry and each column has at least two nonzero entries (exactly two when the respective dofs are in the interior of a {\Face}). Similarly, $\cI_{br}$ is a matrix with the fill-in of $\Pi_{rb}^T$, where all nonzero entries are replaced by 1, each row has exactly one nonzero entry and each column has at least one nonzero entry (exactly one when the respective dofs are in the interior of a {\Face}). Observe that, assuming everything is consistently numbered and ordered, $\cA_{ii} = A_{ii}$ and $\Pi_{ii} = I$. It holds that $A_{ir} = \cA_{is}\cI_{sr}$. This is due to the fact that $A_{ir}$ and $\cA_{is}$ represent the same ``one-sided connections'' (meaning that the ``connections'' are on one side of the {\Faces} $F\in \Phi^H$), but in terms of different indices.

Let $\bv = \Pi_h \bhv = \lb \bhv_i^T, (\Pi_{rs} \bhv_s + \Pi_{rb} \bhv_b)^T \rb^T$. Then,
\begin{align*}
\cI_h \bv &= \cI_h \Pi_h \bhv = \lb \bhv_i^T, [\cI_{sr}(\Pi_{rs} \bhv_s + \Pi_{rb} \bhv_b)]^T, [\cI_{br}(\Pi_{rs} \bhv_s + \Pi_{rb} \bhv_b)]^T \rb^T,\\
\btv = \cI_h \bv - \bhv &= \lb \bzero^T, [\cI_{sr}(\Pi_{rs} \bhv_s + \Pi_{rb} \bhv_b) - \bhv_s]^T, [\cI_{br}(\Pi_{rs} \bhv_s + \Pi_{rb} \bhv_b) - \bhv_b]^T \rb^T,
\end{align*}
and
\begin{align*}
\btv^T \cA \btv &= \frac{1}{\delta}\sum_{T\in\cT^H}\sum_{F\subset \partial T} \lp \restr{\bhv_s}{F} - \restr{\bhv_b}{F} \rp^T D_F \lp \restr{\bhv_s}{F} - \restr{\bhv_b}{F} \rp\\
&\hphantom{=} {} + \sum_{T\in\cT^H} [\restr{(\Pi_{rs} \bhv_s + \Pi_{rb} \bhv_b)}{T} - \restr{\bhv_s}{T}]^T A_{T,rr} [\restr{(\Pi_{rs} \bhv_s + \Pi_{rb} \bhv_b)}{T} - \restr{\bhv_s}{T}],
\end{align*}
where $A_{T,rr}$ is the local version of $A_{rr}$ and it is utilized that locally, as noted above, ``$r$'' and ``$s$'' can be identified. Using that the stiffness matrices can be bounded from above by their diagonals, with some constant $\Lambda > 0$, it follows:
\begin{small}
\begin{equation}\label{eq:bdronlyestimate}
\begin{split}
\btv^T \cA \btv &\le \frac{1}{\delta}\sum_{T\in\cT^H}\sum_{F\subset \partial T} \lp \restr{\bhv_s}{F} - \restr{\bhv_b}{F} \rp^T D_F \lp \restr{\bhv_s}{F} - \restr{\bhv_b}{F} \rp\\
&\hphantom{\le} {} + \Lambda \sum_{T\in\cT^H}\sum_{F\subset \partial T} [\restr{(\Pi_{rs} \bhv_s + \Pi_{rb} \bhv_b)}{F} - \restr{\bhv_s}{F}]^T D_F [\restr{(\Pi_{rs} \bhv_s + \Pi_{rb} \bhv_b)}{F} - \restr{\bhv_s}{F}].
\end{split}
\end{equation}
\end{small}%

Let, for a dof $l$ from the ``$r$'' dofs, $J^s_l$ denote the set of related ``$s$'' edofs and $J^b_l$ -- the set of related bdofs; illustrated in \cref{fig:dofcategories}. That is, $J_l = J^s_l\cup J^b_l$ with $J^s_l$ corresponding to the $l$-th column of $\cI_{sr}$ and $J^b_l$ -- to the $l$-th column of $\cI_{br}$. Furthermore, let $d_l$ be the respective diagonal entry in $A$, $N^b_j$ denote, for all $j \in J^s_l$, the ``$b$'' neighbors (see \cref{fig:dofneighbors}) of $j$, according to the connectivity of $\cA_{sb}$, and $m_l$, $M_l$ are respectively the minimum and maximum values of $\bhv$ on the adofs in $J_l$. Clearly,
\[
d_l \sum_{j\in J_l^s}\lb \frac{1}{\lv J_l \rv}\lp \sum_{k\in J_l^s}(\bhv_s)_k + \sum_{p\in J_l^b}(\bhv_b)_p \rp - (\bhv_s)_j \rb^2 \le \lv J_l^s \rv d_l \lp M_l - m_l \rp^2.
\]
Notice that the adofs corresponding to $M_l$ and $m_l$ are connected via a path, with respect to the connectivity of $\cA_{sb}$, whose length is bounded by $|J_l|$. Following along this path, applying the triangle inequality and \eqref{eq:meshreg}, it holds
\[
\lv J_l^s \rv d_l \lp M_l - m_l \rp^2 \le \varkappa^2 d_l \sum_{j\in J_l^s} \sum_{k\in N_j^b} \lb (\bhv_s)_j - (\bhv_b)_k \rb^2.
\]
Hence,
\[
d_l \sum_{j\in J_l^s}\lb \frac{1}{\lv J_l \rv}\lp \sum_{k\in J_l^s}(\bhv_s)_k + \sum_{p\in J_l^b}(\bhv_b)_p \rp - (\bhv_s)_j \rb^2 \le \varkappa^2 d_l \sum_{j\in J_l^s} \sum_{k\in N_j^b} \lb (\bhv_s)_j - (\bhv_b)_k \rb^2.
\]
Summing over $l$ in the last inequality provides
\begin{align*}
&\sum_{T\in\cT^H}\sum_{F\subset \partial T} [\restr{(\Pi_{rs} \bhv_s + \Pi_{rb} \bhv_b)}{F} - \restr{\bhv_s}{F}]^T D_F [\restr{(\Pi_{rs} \bhv_s + \Pi_{rb} \bhv_b)}{F} - \restr{\bhv_s}{F}]\\
&\le \varkappa^2 \sum_{T\in\cT^H}\sum_{F\subset \partial T} \lp \restr{\bhv_s}{F} - \restr{\bhv_b}{F} \rp^T D_F \lp \restr{\bhv_s}{F} - \restr{\bhv_b}{F} \rp.
\end{align*}
Combining the last estimate and \eqref{eq:bdronlyestimate} implies
\begin{align*}
\btv^T \cA \btv &\le (1+\Lambda\delta\varkappa^2)\frac{1}{\delta}\sum_{T\in\cT^H}\sum_{F\subset \partial T} \lp \restr{\bhv_s}{F} - \restr{\bhv_b}{F} \rp^T D_F \lp \restr{\bhv_s}{F} - \restr{\bhv_b}{F} \rp\\
& \le (1 + \Lambda\delta\varkappa^2)\, \bhv^T \cA \bhv.\qedhere
\end{align*}
\end{proof}

\begin{corollary}\label{cor:cont}
The operator $\Pi_h\col \cE^h \times \cF^h \mapsto \cU^h$ is continuous in the sense
\[
(\Pi_h \bhv)^T A \, \Pi_h \bhv \le 2(2 + \Lambda\delta\varkappa^2)\, \bhv^T \cA \bhv,
\]
for all $\bhv \in \cE^h \times \cF^h$, where the constants are the same as in \cref{lem:approx}.
\end{corollary}
\begin{proof}
Owing to \eqref{eq:stable}, the reverse triangle inequality, and \cref{lem:approx}, it follows:
\begin{align*}
[(\Pi_h \bhv)^T A \, \Pi_h \bhv]^\half - [\bhv^T \cA \bhv]^\half &= [(\cI_h\Pi_h \bhv)^T \cA \, \cI_h\Pi_h \bhv]^\half - [\bhv^T \cA \bhv]^\half\\
&\le [(\cI_h\Pi_h \bhv - \bhv)^T \cA (\cI_h\Pi_h \bhv - \bhv)]^\half\\
&\le \sqrt{1 + \Lambda\delta\varkappa^2} [\bhv^T \cA \bhv]^\half.\qedhere
\end{align*}
\end{proof}

\begin{figure}
\centering
\includegraphics[width=1.00\textwidth]{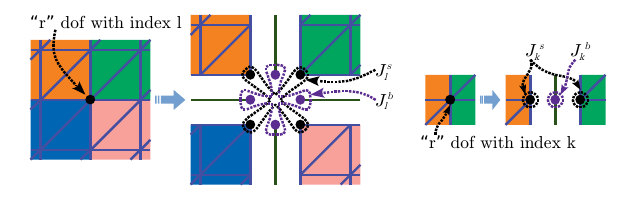}
\caption[]{An illustration of the sets of adofs, associated with an ``$r$'' dof, used in the proof of \cref{lem:approx}.}\label{fig:dofcategories}
\end{figure}

\begin{figure}
\centering
\includegraphics[width=0.8\textwidth]{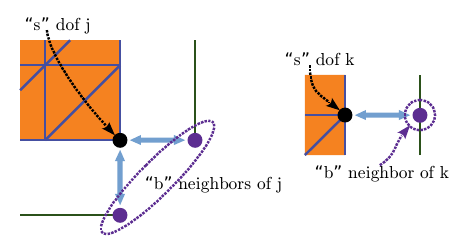}
\caption[]{An illustration of the ``$b$'' neighbors, according to the connectivity of $\cA_{sb}$ in \eqref{eq:cA}, of an ``$s$'' dof. These are designated as $N_j^b$ and $N_k^b$ and used in the proof of \cref{lem:approx}.}\label{fig:dofneighbors}
\end{figure}

\begin{remark}\label{rem:const}
The independence of the constants, in the bounds in \cref{lem:approx,cor:cont}, on the coefficient $\kappa$ in \eqref{eq:pde} holds, since the coefficient information is contained in $D_F$ in \eqref{eq:localip}. In fact, the use of the diagonal entries of the global $A$ is crucial. However, the particular argument does not exclude dependence of the constants on the order of the finite element spaces. This is due to $\Lambda$ in \eqref{eq:bdronlyestimate} depending on the maximum number of dofs in an element. This can be mitigated in the analysis, if necessary, via ``strengthening'' the interface terms in \eqref{eq:localip} by using a properly scaled version of the diagonal of $A$ (e.g., adjusting the value of $\delta$), or employing a so called weighted $\ell_1$-smoother (see \cite{2011Smoothers}) in lieu of the diagonal of $A$ for the interface term in \eqref{eq:localip}, which would lead to $\Lambda = 1$. Note that, interestingly, the numerical results in \cref{ssec:highorder} (where the unweighted diagonal of $A$ is used) do not demonstrate a practical necessity for adjusting the interface terms, when increasing the polynomial order. This indicates that a more intricate proof, that does not depend on the constant $\Lambda$ (i.e., on the polynomial order), may be possible. Currently, this is an open question.
\end{remark}

Finally, based on the above properties, the optimality of the auxiliary space preconditioners can be established.

\begin{theorem}[spectral equivalence]\label{thm:ipopt}
Let the smoother $M$ satisfy the property that $M + M^T - A$ is SPD.
Assume that $\cB$ in \eqref{eq:Badd} and \eqref{eq:Bmult} is a spectrally equivalent preconditioner for the IP matrix $\cA$, in the sense that there exist positive constants $\alpha$ and $\beta$, such that
\begin{equation}\label{eq:spectequiv}
\alpha^\mone\, \bhv^T \cA \bhv \le \bhv^T \cB \bhv \le \beta\, \bhv^T \cA \bhv,\quad\forall\bhv \in \cE^h \times \cF^h.
\end{equation}
Then, the additive and multiplicative auxiliary space preconditioners, $B_{\mathrm{add}}$ in \eqref{eq:Badd} and $B_{\mathrm{mult}}$ in \eqref{eq:Bmult}, are spectrally equivalent to $A$ in \eqref{eq:H1linsys}.
\end{theorem}
\begin{proof}
Owing to \eqref{eq:spectequiv}, \eqref{eq:stable} and \cref{cor:cont} provide, respectively,
\begin{alignat}{3}
(\cI_h \bv)^T \cB\, \cI_h \bv &\le \beta\,\bv^T A \bv,\quad&&\forall \bv \in \cU^h,\label{eq:stableB}\\
(\Pi_h \bhv)^T A \, \Pi_h \bhv &\le 2\alpha(2 + \Lambda\delta\varkappa^2)\, \bhv^T \cB \bhv,\quad&&\forall\bhv \in \cE^h \times \cF^h.\label{eq:contB}
\end{alignat}

First, consider the ``fictitious space preconditioner'' $\wtB^\mone = \Pi_h \cB^\mone \Pi_h^T$ for $A$. It is clearly SPD, when $\cB$ is SPD (implied by \eqref{eq:spectequiv}), due to the full row rank of $\Pi_h$. By \eqref{eq:approx}, the Cauchy–Schwarz inequality, and \eqref{eq:stableB}, it follows:
\begin{align*}
\bv^T A \bv &= (\cI_h \bv)^T \cB \cB^\mone \Pi_h^T A \bv \le [(\cI_h \bv)^T \cB\, \cI_h \bv]^\halff [\bv^T A \Pi_h \cB^\mone \cB \cB^\mone \Pi_h^T A \bv]^\halff\\
&\le \sqrt{\beta} [\bv^T A \bv]^\halff [\bv^T A \wtB^\mone A \bv]^\halff.
\end{align*}
Thus, $\bv^T A \bv \le \beta\, \bv^T A \wtB^\mone A \bv$, showing
\begin{equation}\label{eq:fictstable}
\bv^T A^\mone \bv \le \beta\, \bv^T \wtB^\mone \bv,\quad\forall \bv \in \cU^h,
\end{equation}
which is equivalent to $\bv^T \wtB \bv \le \beta\, \bv^T A \bv$. Conversely, owing to the Cauchy–Schwarz inequality and \eqref{eq:contB}, (denoting $\gamma = 2\alpha(2 + \Lambda\delta\varkappa^2)$) it follows:
\begin{small}
\begin{align*}
\bv^T A \wtB^\mone A \bv &\le [\bv^T A \bv]^\halff [\bv^T A \wtB^\mone A \wtB^\mone A \bv]^\halff\\
&= [\bv^T A \bv]^\halff [ (\Pi_h \cB^\mone \Pi_h^T A \bv)^T A \Pi_h \cB^\mone \Pi_h^T A \bv]^\halff\\
&\le \sqrt{\gamma}[\bv^T A \bv]^\halff [ \bv^T A \Pi_h \cB^\mone  \cB \cB^\mone \Pi_h^T A \bv]^\halff = \sqrt{\gamma}[\bv^T A \bv]^\halff [ \bv^T A \wtB^\mone A \bv]^\halff.
\end{align*}
\end{small}%
Whence, $\bv^T A \wtB^\mone A \bv \le 2\alpha(2 + \Lambda\delta\varkappa^2)\, \bv^T A \bv$ and
\begin{equation}\label{eq:fictcont}
\bv^T \wtB^\mone \bv \le 2\alpha(2 + \Lambda\delta\varkappa^2)\, \bv^T A^\mone \bv,\quad\forall \bv \in \cU^h,
\end{equation}
or, equivalently, $\bv^T A \bv \le 2\alpha(2 + \Lambda\delta\varkappa^2)\, \bv^T \wtB \bv$.

Next, consider $B^\mone_{\mathrm{add}}$ in \eqref{eq:Badd}. In view of the positive definiteness of the smoother in \eqref{eq:Badd}, \eqref{eq:fictstable} implies
\[
\bv^T A^\mone \bv \le \beta\, \bv^T B^\mone_{\mathrm{add}} \bv,\quad\forall \bv \in \cU^h,
\]
equivalently, $\bv^T B_{\mathrm{add}} \bv \le \beta\, \bv^T A \bv$. Conversely, using that $M + M^T - A$ is SPD and the fact that this is equivalent to $2\ol{M} - A$ being SPD \cite{VassilevskiMG}, it holds that 
\begin{equation}\label{eq:olM}
\bv^T A \bv \le 2\,\bv^T \ol{M} \bv,\quad\forall \bv \in \cU^h.
\end{equation}
Similarly, in case $M$ is symmetric and $M^\mone$ is used in \eqref{eq:Badd}, it holds $\bv^T A \bv \le 2\,\bv^T M \bv$. This, combined with \eqref{eq:fictcont}, in view of \eqref{eq:Badd}, implies
\[
\bv^T B^\mone_{\mathrm{add}} \bv \le 2[\alpha(2 + \Lambda\delta\varkappa^2) + 1]\, \bv^T A^\mone \bv,\quad\forall \bv \in \cU^h,
\]
equivalently, $\bv^T A \bv \le 2[\alpha(2 + \Lambda\delta\varkappa^2) + 1]\, \bv^T B_{\mathrm{add}} \bv$.

Finally, consider $B^\mone_{\mathrm{mult}}$ in \eqref{eq:Bmult}. Note that $M + M^T - A$ being SPD is equivalent to the stationary iteration with $M$ (and $M^T$) being $A$-convergent (convergent in the norm $\lV\cdot\rV_A$, induced by the matrix $A$), i.e., $\lv I - A^\halff M^\mone A^\halff \rv = \lV I - M^\mone A \rV_A = \lv I - A^\halff M^{-T} A^\halff \rv = \lV I - M^{-T} A \rV_A < 1$, where $\lv\cdot \rv$ is the Euclidean norm; see, e.g., \cite{VassilevskiMG}. Then, using \eqref{eq:fictcont},
\begin{small}
\begin{align*}
&\hphantom{={}}\sup_{\bv \ne \bzero} \frac{\bv^T (I - M^{-T} A) \wtB^\mone (I - A M^{-1}) \bv}{\bv^T A^\mone \bv} = \sup_{\bv \ne \bzero} \frac{\bv^T A^\half (I - M^{-T} A) \wtB^\mone (I - A M^{-1}) A^\half \bv}{\bv^T \bv}\\
&= \sup_{\bv \ne \bzero} \frac{\bv^T (I - A^\half M^{-T} A^\half)A^\half \wtB^\mone A^\half (I - A^\half M^{-1}A^\half) \bv}{\bv^T \bv} \\
&= \lv (I - A^\half M^{-T} A^\half)A^\half \wtB^\mone A^\half (I - A^\half M^{-1}A^\half) \rv \le \lv I - A^\half M^{-1}A^\half \rv^2 \lv A^\half \wtB^\mone A^\half \rv\\
&\le \lv A^\half \wtB^\mone A^\half \rv = \sup_{\bv \ne \bzero}\frac{\bv^T A^\half \wtB^\mone A^\half \bv}{\bv^T \bv} = \sup_{\bv \ne \bzero}\frac{\bv^T \wtB^\mone \bv}{\bv^T A^\mone \bv} \le 2\alpha(2 + \Lambda\delta\varkappa^2).
\end{align*}
\end{small}%
This and \eqref{eq:olM}, in view of \eqref{eq:Bmult}, show
\[
\bv^T B^\mone_{\mathrm{mult}} \bv \le 2[\alpha(2 + \Lambda\delta\varkappa^2) + 1]\, \bv^T A^\mone \bv,\quad\forall \bv \in \cU^h,
\]
equivalently, $\bv^T A \bv \le 2[\alpha(2 + \Lambda\delta\varkappa^2) + 1]\, \bv^T B_{\mathrm{mult}} \bv$. Conversely, using the definition \eqref{eq:Bmult} of $B^\mone_{\mathrm{mult}}$, \eqref{eq:fictstable}, and the positive definiteness of $\ol{M}$, and taking $\xi = \max\{\beta, 1 \}$, it holds
\begin{align*}
\bv^T B^\mone_{\mathrm{mult}} \bv &= \bv^T \ol{M}^\mone \bv + \bv^T (I - M^{-T}A) \wtB^\mone (I - A M^\mone) \bv\\
&\ge \bv^T \ol{M}^\mone \bv + \xi^\mone\,\bv^T (I - M^{-T}A) A^\mone (I - A M^\mone) \bv\\
&= \bv^T \ol{M}^\mone \bv + \xi^\mone\,\bv^T A^{-\half} (I - A^\half M^{-T} A^\half)(I - A^\half M^\mone A^\half) A^{-\half}\bv\\
&= \bv^T \ol{M}^\mone \bv + \xi^\mone\,\bv^T A^{-\half} (I - A^\half \ol{M}^\mone A^\half) A^{-\half}\bv\\
&= (1-\xi^\mone)\,\bv^T \ol{M}^\mone \bv + \xi^\mone\,\bv^T A^\mone \bv \ge \xi^\mone\,\bv^T A^\mone \bv,
\end{align*}
or, equivalently, $\bv^T B_{\mathrm{mult}} \bv \le \max\{\beta, 1 \}\, \bv^T A \bv$.
\end{proof}

\begin{remark}
There is a couple of additional assumptions on the smoother, $M$, in \cite[Theorem 7.18]{VassilevskiMG}. However, they are not necessary in \cref{thm:ipopt}, due to the exactness of the approximation \eqref{eq:approx}. Only the basic property of $A$-convergence of the iteration with $M$ (i.e., $M + M^T - A$ being SPD) is assumed. In fact, \eqref{eq:fictstable} and \eqref{eq:fictcont} show, counting on \eqref{eq:approx}, that $\wtB$ alone is spectrally equivalent to $A$, without the necessity of additional smoothing in the fine-scale setting here. Nevertheless, when coarse auxiliary spaces (described in \cref{ssec:ipcoarseform}) are utilized smoothing is necessary and a smoother, $M$, is outlined in \cref{sec:smoother}. Note that \eqref{eq:fictstable} and \eqref{eq:fictcont} essentially represent, based on the surjectivity of $\Pi_h$, the so called ``fictitious space lemma''; cf. \cite[Theorem 10.1]{2007DD}, \cite[Theorem 2.1]{1996AuxiliarySpace}
\end{remark}

\begin{remark}\label{rem:general}
Notice that the construction of the space $\cE^h\times \cF^h$ from $\cU^h$ and the proofs of \cref{lem:approx,cor:cont,thm:ipopt} are abstract, algebraic, and general in nature, using only generic properties of SPD and SPSD matrices. The degrees of freedom can be viewed in a general abstract sense, using only that they can be related to mesh entities (like elements, faces, edges, and vertices). A careful inspection of the proofs shows that the particular form of the model problem \eqref{eq:pde}, or \eqref{eq:weakform}, and its properties (particularly, that it is an elliptic PDE) are not utilized. Thus, the IP reformulation is applicable and its spectral equivalence properties are maintained for quite general SPD systems (i.e., convex quadratic minimization problems) that can be associated with appropriate local SPSD versions. In a general setting, the IP reformulation can be constructed as in \eqref{eq:localip}, providing the element-by-element assembly property for a given problem, via replacing $A_T$ and $D_F$ with the corresponding matrices associated with the respective convex quadratic minimization of interest. Consequently, the approach is not limited to preconditioning systems necessarily coming from an $H^1$-conforming, per se, finite element method and potentially it can even be applicable to systems that may not be associated with any finite element method. For example, in the case of Raviart-Thomas spaces \cite{BoffiMFE}, constructing $\cE^h\times \cF^h$ involves cloning all dofs associated with the fluxes on the faces constituting a {\Face} and, for N{\'e}d{\'e}lec spaces, the dofs associated with faces and their respective edges (e.g., in 3D) are cloned. Essentially, the overall procedure remains the same and the theoretical results (together with their proofs) carry over.
\end{remark}

\subsection{Coarse formulation}
\label{ssec:ipcoarseform}

Combining the considerations in \cref{ssec:ipformulation} and the end of \cref{ssec:spaces}, an IP formulation on a coarse pair of spaces, $\cE^H \times \cF^H$, can be obtained. Namely, having the prolongation matrix $\cP\col \cE^H \times \cF^H \mapsto \cE^h \times \cF^h$, the coarse IP matrix can be obtained via a standard ``RAP'' procedure. That is, $\cA^H = \cP^T \cA \cP$ and, locally, $\cA^H_T = \cP^T_T \cA_T \cP_T$, where $\cP_T$ is a local version of $\cP$, with columns -- the basis vectors associated with $T$ and all $F \subset \partial T$.

Alternatively, $\cA^H_T$ can be computed directly, using \eqref{eq:localip} with $\bv_e$ in the span of $\set{\bq_{T,i}}_{i=1}^{m_T}$ and $\bv_b$ in the subspace spanned by $\set{\bq_{F,i}}_{i=1}^{m_F}$, for all $F \subset \partial T$, while $\cA^H$ can be assembled from the local matrices, $\cA^H_T$. Respective consistent, with $\cA^H_T$, relations {\Element}_{\Edof}, {\Element}_{\Bdof}, and, consequently, {\Element}_{\Adof} are obtainable. This practically removes the necessity to construct the fine-scale space, $\cE^h \times \cF^h$, and matrices -- $\cA$ and $\cA_T$. Furthermore, the operator $\Pi_H\col \cE^H \times \cF^H \mapsto \cU^h$ is defined as $\Pi_H = \Pi_h \cP$ and can be constructed directly. Its continuity,
\[
(\Pi_H \bhv_c)^T A \, \Pi_H \bhv_c \le 2(2 + \Lambda\delta\varkappa^2) \bhv_c^T \cA^H \bhv_c,\quad \forall\bhv_c \in \cE^H \times \cF^H,
\]
follows immediately from \cref{cor:cont}. The coarse-space versions of the auxiliary space preconditioners for $A$, where $\cB^H$ is a symmetric preconditioner for $\cA^H$, are
\begin{equation}\label{eq:BH}
\begin{split}
(B_{\mathrm{add}}^H)^\mone &= \ol{M}^\mone + \Pi_H (\cB^H)^\mone \Pi_H^T,\\
(B_{\mathrm{mult}}^H)^\mone &= \ol{M}^\mone + (I - M^{-T}A)\Pi_H (\cB^H)^\mone \Pi_H^T (I - A M^\mone).
\end{split}
\end{equation}

A particular approach, considered here, for obtaining a coarse basis and constructing $\cE^H \times \cF^H$ is by solving local generalized eigenvalue problems, cf. \cite{2012SAAMGe}, of the type
\begin{equation}\label{eq:eval}
A_T \bq = \lambda D_T \bq,
\end{equation}
where $D_T$ is the diagonal of $A_T$ and the eigenvalues are ordered $\lambda_1 \le \dots \le \lambda_{n_T}$. The first $m_T$ ($1\le m_T < n_T$) eigenvectors of \eqref{eq:eval}, constitute the $D_T$-orthogonal basis $\set{\bq_{T,i}}_{i=1}^{m_T}$ associated with $T$. Particularly, let $\theta \in (0,1)$ be given and $m_T$ satisfy $\lambda_{m_T} \le \theta \lambda_{n_T}$ and $\lambda_{m_T+1} > \theta \lambda_{n_T}$. For each $F\in \Phi^H$, collect $\set{\restr{\bq_{T_+,i}}{F}}_{i=1}^{m_{T_+}}$ and $\set{\restr{\bq_{T_-,i}}{F}}_{i=1}^{m_{T_-}}$ from its adjacent {\Elements} $T_+$ and $T_-$. After performing SVD to filter out any linear dependence, the $\ell_2(F)$-orthogonal basis $\set{\bq_{F,i}}_{i=1}^{m_F}$, associated with $F$, is obtained. Notice that this procedure maintains the relation, similarly to the fine-scale setting, that $\cF^H$ is the trace-space (on {\Faces}) for the functions in $\cE^H$.

\begin{remark}
In general, as far as the feasibility of the approach is concerned, $\cF^H$ does not need to be precisely the trace-space for $\cE^H$. Particularly, it can be a proper subspace of the trace-space (on {\Faces}) for the functions in $\cE^H$, potentially providing additional reduction, when employing static condensation (\cref{ssec:ipstatcond}). The study of the choices of $\cF^H$ and the effect this may have on the properties of the preconditioners is a subject of future work.
\end{remark}

\begin{remark}
An alternative choice is to utilize polynomial bases for $\cE^H$ and $\cF^H$ of order lower than the order of $\cE^h$ and $\cF^h$. This paper concentrates on coarse basis vectors obtained via the above described generalized eigenvalue problems.
\end{remark}

\subsection{Static condensation}
\label{ssec:ipstatcond}

Consider, to simplify the presentation, fine-scale spaces and matrices for the IP method -- $\cE^h \times \cF^h$, $\cA$, and $\cA_T$. The idea is, in lieu of directly preconditioning $\cA$ in $B_\mathrm{add}^\mone$ and $B_\mathrm{mult}^\mone$, to eliminate all edofs in $\cA$ and precondition the resulting Schur complement, expressed only on the bdofs. This procedure is referred to as \emph{static condensation}, since it involves ``condensing'' the formulation on the interfaces.

Namely, using the block factorization (cf. \eqref{eq:cA})
\[
\cA =
\begin{bmatrix}
\cA_{ee} & \\
\cA_{be} &\cS
\end{bmatrix}
\begin{bmatrix}
I & \cA_{ee}^\mone \cA_{eb}  \\
 & I
\end{bmatrix},
\]
where $\cS = \cA_{bb} - \cA_{be} \cA_{ee}^\mone \cA_{eb}$ is the respective SPD Schur complement, a preconditioner for $\cA$ is obtained as follows:
\begin{equation}\label{eq:Bsc}
\cB_\mathrm{sc}^\mone =
\begin{bmatrix}
I & -\cA_{ee}^\mone\cA_{eb} \\
 & I
\end{bmatrix}
\begin{bmatrix}
\cA_{ee}^\mone & \\
-S^\mone \cA_{be} \cA_{ee}^\mone & S^\mone
\end{bmatrix},
\end{equation}
where $S$ is a SPD preconditioner for $\cS$, providing a SPD $\cB_\mathrm{sc}$. Observe that computing the action of $\cB_\mathrm{sc}^\mone$ involves applying $S^\mone$ once and $\cA_{ee}^\mone$ twice: first, during the ``elimination'' stage, represented by the second factor in \eqref{eq:Bsc}, that provides the argument for $S^\mone$; second, during the ``backward substitution'' stage, represented by the first factor in \eqref{eq:Bsc}, that updates the edof portion of the solution. If $S$ is spectrally equivalent to $\cS$, then $\cB_\mathrm{sc}$ is spectrally equivalent to $\cA$.

Clearly, obtaining $\cA_{ee}^\mone$, or an appropriate approximation thereof, is necessary for computing both $\cS$ and the action of $\cB_\mathrm{sc}^\mone$. As mentioned, no edofs are coupled across {\Elements}. That is, $\cA_{ee}$ is block-diagonal -- $\cA_{ee} = \diag(\cA_{T,ee})_{T\in\cT^H}$; see \eqref{eq:cAT}. Thus, computing $\cA_{ee}^\mone$, for obtaining $\cS$ and the action of $\cB_\mathrm{sc}^\mone$, only involves local work on {\Elements}. Moreover, consider the local, for $T$ and the respective $\cA_T$, cf. \eqref{eq:cAT}, SPSD Schur complements $\cS_T = \cA_{T,bb} - \cA_{T,be} \cA_{T,ee}^\mone \cA_{T,eb}$, formulated on the bdofs associated with $T$ (more precisely, with $\partial T$). Notice that the $\cS_T$ matrices provide a local structure associated with $\cS$, since $\cS$ can be assembled from $\cS_T$. Hence, while $\cS_T$ are generally dense, $\cS$ possesses a typical sparsity pattern associated with matrices coming from finite element methods. The local structures of $\cA$ and $\cS$, respectively provided by $\cA_T$ and $\cS_T$, facilitate the application of AMGe methods for preconditioning $\cA$ and $\cS$, as described in \cref{sec:nonconfprecond}.

Note that everything is general and also applies in the coarse-scale setting, of \cref{ssec:ipcoarseform}, by instead utilizing $\cE^H \times \cF^H$, $\cA^H$, and $\cA_T^H$, providing the Schur matrices $\cS^H$ and $\cS_T^H$, as well as the preconditioners $S^H$ for $\cS^H$ and $(\cB_\mathrm{sc}^H)^\mone$ for $\cA^H$.

\section{The smoother \texorpdfstring{$M$}{M}}
\label{sec:smoother}

A particular smoother $M$, for $A$, which is a part of the auxiliary space preconditioners proposed in this paper, is shortly described now. Particularly, a polynomial smoother is utilized based on the Chebyshev polynomial of the first kind.

For a given integer $\nu \ge 1$, consider the polynomial of degree $3\nu + 1$ on $[0,1]$
\[
p_\nu(t) = \lp 1 - T_{2\nu+1}^2(\sqrt{t}) \rp (-1)^\nu \frac{1}{2\nu+1} \frac{T_{2\nu+1}(\sqrt{t})}{\sqrt{t}},
\]
satisfying $p_\nu(0) = 1$, where $T_l(t)$ is the Chebyshev polynomial of the first kind on $[-1, 1]$. Then, $M$ is defined as $M^\mone = [I - p_\nu(b^\mone D^\mone A)]A^\mone$ or, equivalently, $I - M^\mone A = p_\nu(b^\mone D^\mone A)$, where $b = \cO(1)$ is a parameter satisfying $\bv^T A \bv \le b\, \bv^T D \bv$, for all $\bv \in \cU^h$, and $D$ is the diagonal of $A$ (or another appropriate diagonal matrix). Note that $M$ is SPD and the action of such a polynomial smoother is computed via $3\nu+1$ Jacobi-type iterations, using the roots of the polynomial, which makes it convenient for parallel computations; see \cite[Section 4.2.2]{MSthesis}. 

In practice, $D$, in the definition of $M$, can be replaced by a diagonal weighted $\ell_1$-smoother like $W = \diag(w_i)_{i=1}^{\dim(\cU^h)}$, where $w_i = \sum_{j=1}^{\dim(\cU^h)} \lv a_{ij} \rv \sqrt{a_{ii}/a_{jj}}$. Such a choice allows setting $b = 1$.

For more information on the subject, consult \cite{VassilevskiMG,2011Smoothers,2012SAAMGe,2012ConvSAAMG,1999TGElasticity,2012xinv}.

\section{AMGe for the nonconforming formulations}
\label{sec:nonconfprecond}

A procedure for constructing AMGe hierarchies and preconditioners, for SPD matrices and applicable to the auxiliary space problems, is outlined in this section. The approach is related to the ideas in \cite{2007AMGe,2008AMGe,2012SAAMGe}. First, the method is described in a general abstract setting, which requires reusing some of the notions and notation, introduced in \cref{sec:basic}, in a slightly different context. Then, details concerning the specifics of applying the approach in the auxiliary space setting are discussed, including the mapping between the terminology here and in \cref{sec:basic}. This should avoid any confusion between the notions and they should not get mixed.

\subsection{General description}
\label{ssec:generalsaamge}

Let a collection $\fT^h = \set{\tau}$ of elements and an associated set of dofs $\fD^h$, on $\fT^h$, be given. Here, $\fT^h$ and $\fD^h$ supply the respective relations element_element and element_dof. Additionally, SPSD element matrices $A_\tau$, for $\tau \in \fT^h$, are obtainable, whose global assembly, in accordance with element_dof, provides a SPD matrix $A$. This constitutes the given data and the goal is to construct a preconditioner for $A$. Typically, such data comes from meshes, finite element spaces, and weak formulations, as is the case in this paper, but the approach can be evoked with data that is generally not necessarily associated with finite element methods.

Using element_element and, e.g., a graph partitioner, a collection $\fT^H = \set{T}$ of {\Elements} is obtained, where the {\Elements}, $T$, are non-overlapping and connected, in terms of element_element, sets of elements, $\tau$. This is expressed via the relation {\Element}_element. As before, compute
\[
\text{\Element_dof} = \text{\Element_element} \times \text{element_dof},
\]
which represents the {\Elements} as sets of dofs. An intersection procedure\footnote{This is fundamentally the same procedure as the one identifying {\Faces} in \cref{ssec:meshagg}.} over {\Element}_dof generates so called \emph{minimal intersection sets} (MISes); see \cite{VassilevskiMG,2007AMGe,2008AMGe}. In more detail, MISes are (non-overlapping) equivalence classes of dofs with respect to the relation that two dofs are equivalent if they belong to identical sets of {\Elements}, in accordance with $(\text{\Element_dof})^T$. The collection of MISes partitions the set of dofs, $\fD^h$, and, for every $T$, there is a subcollection of MISes that partitions the dofs in $T$, as determined by \Element_dof. The intersection procedure constructs the relations \textsc{mis}_dof and {\Element}_\textsc{mis}, where
\[
\text{{\Element}_\textsc{mis}} = \text{{\Element}_dof} \times (\text{\textsc{mis}_dof})^T.
\]

Next, the derivation of a coarse basis, as a linearly independent set of vectors defined on dofs, is described. Based on \Element_element (also, utilizing element_dof and \Element_dof), assemble local, on $T$, SPSD matrices $A_T$ from the element matrices $A_\tau$, for $\tau \in T$. Solve local generalized eigenvalue problems
\begin{equation}\label{eq:seval}
A_T \bq = \lambda D_T \bq,
\end{equation}
where $D_T$ is the diagonal of $A_T$ and the eigenvalues are ordered $\lambda_1 \le \dots \le \lambda_{n_T}$. Take the first $m_T$ ($1\le m_T < n_T$) eigenvectors of \eqref{eq:seval} forming the $D_T$-orthogonal set $\set{\bq_{T,i}}_{i=1}^{m_T}$ associated with $T$. Particularly, let $\theta_s \in (0,1)$ be given and $m_T$ satisfy $\lambda_{m_T} \le \theta_s \lambda_{n_T}$ and $\lambda_{m_T+1} > \theta_s \lambda_{n_T}$. Then, for each MIS $\sM$, collect $\set{\bq_{T,i}}_{i=1}^{m_T}$ from all $T$ associated with $\sM$, via $(\text{{\Element}_\textsc{mis}})^T$, and restrict them to $\sM$ based on the relations \textsc{mis}_dof and {\Element}_dof. After performing SVD to filter out any linear dependence, the $\ell_2(\sM)$-orthogonal basis $\set{\bq_{\sM,i}}_{i=1}^{m_\sM}$, associated with $\sM$, is obtained and organized as the columns of a local, on $\sM$, prolongation matrix $P_{\sM} = [\bq_{\sM,1}, \cdots, \bq_{\sM,m_\sM}]$. The resulting global prolongation matrix takes the form
\[
P =
\begin{bmatrix}
P_{\sM_1} &  & \\
& \ddots & \\
& & P_{\sM_{n_a}}
\end{bmatrix}.
\]
Similarly, for $T \in \fT^H$ and all $\sM$ related to $T$, via {\Element}_\textsc{mis}, collect all $P_{\sM}$ to construct a local, on $T$, prolongation matrix $P_T$.

By building $P$, a coarse space is constructed, where the columns of $P$ form the basis. Identify the coarse degrees of freedom, designated by ``\Dofs'' and constituting the set $\fD^H$, with the columns of $P$, i.e., with the basis vectors, and define the relation {\Dof}_dof as a representation of the sparsity pattern of $P^T$. Consequently, a {\Dof} is related to a dof if and only if the respective coarse basis vector is supported on that dof. Similarly, an {\Element} is related to a {\Dof}, determining {\Element}_{\Dof}, whenever the respective coarse basis vector is supported on any dof of the {\Element}. That is,
\[
\text{{\Element}_{\Dof}} = \text{{\Element}_dof} \times (\text{{\Dof}_dof})^T.
\]
Notice that this corresponds to the columns of $P_T$. Hence, coarse SPSD element matrices are consistently obtained via local ``RAPs'' -- $A_{T,c} = P_T^T A_T P_T$ -- and the global SPD coarse matrix is $A_c = P^T A P$, which can be assembled from $A_{T,c}$, for $T \in \fT^H$, in accordance with {\Element}_{\Dof}. Furthermore, {\Elements} are related via
\begin{align*}
\text{{\Element}_{\Element}} = \text{{\Element}_element} \times \text{element_element}\\
\times (\text{{\Element}_element})^T.
\end{align*}

By assigning
\begin{gather*}
\fT^h \leftarrow \fT^H,\;\fD^h \leftarrow \fD^H,\;\text{element_element} \leftarrow \text{{\Element}_{\Element}},\\
\text{element_dof} \leftarrow \text{{\Element}_{\Dof}},\; \set{A_\tau} \leftarrow \set{A_{T,c}},
\end{gather*}
the above procedure can be executed recursively, building a hierarchy of meshes, spaces, bases, degrees of freedom, element matrices, global matrices, relations, and transition operators between the spaces. To obtain multilevel cycles as preconditioners, only remains to appoint a ``relaxation'' (``smoothing'') method on each level. To that purpose, the polynomial smoother in \cref{sec:smoother} is employed with an integer parameter $\nu_s \ge 1$. This algorithm is implemented, including parallel, in SAAMGE \cite{saamge}.

\begin{remark}\label{rem:shortsaamge}
Observe that the approach here can be invoked with given data: $\fT^h$, $\fD^h$, element_element, $\fT^H$ (in the form of {\Element}_element), {\Element}_dof, and $\set{A_T\where T \in \fT^H}$. In that case, the procedures for obtaining these structures are skipped and the precursors (element_dof and $\set{A_\tau\where \tau \in \fT^h}$), which, in fact, may not be available, are not needed.
\end{remark}

\subsection{On the nonconforming formulations}

Specifics, concerning the application of the method outlined in \cref{ssec:generalsaamge} in the context of the nonconforming formulations, are addressed now.

First, consider the fine-scale IP setting, involving $\cE^h \times \cF^h$, $\cA$, and $\cA_T$. Since agglomeration is already performed to form the IP problem, without employing a coarsening process for the spaces, it is natural to reuse the same agglomeration and assign $\fT^h \leftarrow \cT^h$ (using its element_element), the set of adofs (associated with $\cE^h \times \cF^h$) as $\fD^h$, $\fT^H \leftarrow \cT^H$ with the respective {\Element}_element and $\text{{\Element}_dof} \leftarrow \text{{\Element}_adof}$; see \cref{ssec:meshagg,ssec:ipformulation}. In this case, the respective local, for $T\in \fT^H$, agglomerate matrices are the already assembled IP matrices $\set{\cA_T\where T \in \fT^H}$. Then, in view of \cref{rem:shortsaamge}, the general scheme in \cref{ssec:generalsaamge} produces a hierarchy of spaces and a multilevel preconditioner for $\cA$.

Notice that the outlined methodology automatically generates a hierarchy of space pairs $\cE^{H_{l}} \times \cF^{H_{l}}$, for $l = 0,\dots,n_\ell$ and $H_{0} = h$, and the IP method formulated on these spaces via a variational (``RAP'') procedure. This is due to the MISes' construction, using {\Element}_adof, automatically reidentifying the {\Faces} and separating edofs from bdofs, which eventually results in obtaining separate bases associated with {\Elements} and interfaces. Thus, once the space pair $\cE^h \times \cF^h$ is constructed as discontinuous, the whole hierarchy maintains the same general complexion.

Comparing $\cE^{H_{1}} \times \cF^{H_{1}}$, constructed here by a single coarsening, with the space pair in \cref{ssec:ipcoarseform}, resemblance can be seen. While the notion of {\Element} is the same, a major difference is that $\cE^{H_{1}} \times \cF^{H_{1}}$ uses basis vectors obtained via generalized eigenvalue problems for the local IP matrices $\cA_T$, as defined by \eqref{eq:localip}, whereas \cref{ssec:ipcoarseform} utilizes $A_T$ -- the local representations of \eqref{eq:weakform}. Moreover, the final preconditioners for $A$ have clear differences. Namely, $B_\mathrm{add}$ in \eqref{eq:Badd} and $B_\mathrm{mult}$ in \eqref{eq:Bmult} employ a fine-scale auxiliary space, $\cE^h \times \cF^h$, and a multigrid preconditioner, $\cB$, for $\cA$ acts on the respective fine level via relaxation. In contrast, $B_\mathrm{add}^H$ and $B_\mathrm{mult}^H$, in \eqref{eq:BH}, directly utilize a coarse auxiliary space, $\cE^H \times \cF^H$, and a multigrid preconditioner, $\cB^H$, for $\cA^H$ performs no action on the fine level. In the latter case, the method of \cref{ssec:generalsaamge} can still be employed to precondition $\cA^H$ from \cref{ssec:ipcoarseform}. This is described next.

Since one coarsening is readily accomplished in \cref{ssec:ipcoarseform}, the coarse level there acts as a fine one for the preconditioner here. That is, assign $\fT^h \leftarrow \cT^H$, $\text{element_element} \leftarrow \text{\Element_\Element}$, the set of {\Adofs} (associated with $\cE^H \times \cF^H$) as $\fD^h$, and $\text{element_dof} \leftarrow \text{{\Element}_{\Adof}}$; see \cref{ssec:meshagg,ssec:ipcoarseform}. Also, the available coarse IP matrices $\set{\cA_T^H\where T \in \cT^H}$ act as fine-level element matrices for the solver here. Then, the full procedure in \cref{ssec:generalsaamge} can be invoked, obtaining a multilevel preconditioner for $\cA^H$.

Finally, the AMGe method of \cref{ssec:generalsaamge} can be applied to the ``condensed'' formulations of \cref{ssec:ipstatcond}, since the approach in \cref{ssec:ipstatcond} provides a local structure. Particularly, the static condensation method eliminates edofs, leaving only bdofs. That is, the only remaining space is $\cF^h$ and a ``condensed'' formulation on that space. The procedure is analogous and obtainable from the above discussion by simply replacing: adof by bdof, $\cA$ by $\cS$, and $\cA_T$ by $\cS_T$. Then, the agglomeration and MISes' construction coarsen the interfaces and produce basis vectors to generate a space hierarchy $\cF^{H_{l}}$ and a multilevel preconditioner for $\cS$, where ``RAP'' provides ``condensed'' formulations on each level. Similarly, having in mind \cref{ssec:ipcoarseform}, a multigrid preconditioner for $\cS^H$ can be obtained starting from $\cF^H$ by using the above procedure and replacing: {\Adof} by {\Bdof}, $\cA^H$ by $\cS^H$, and $\cA_T^H$ by $\cS_T^H$. Note that, as before, starting with $\cF^h$ reuses the agglomeration from the IP method, whereas starting with $\cF^H$ interprets those agglomerates as fine-level elements.

\section{Numerical examples}
\label{sec:numerical}

\begin{figure}
\subfloat[][Coefficient in 3D]{\includegraphics[width=0.485\textwidth]{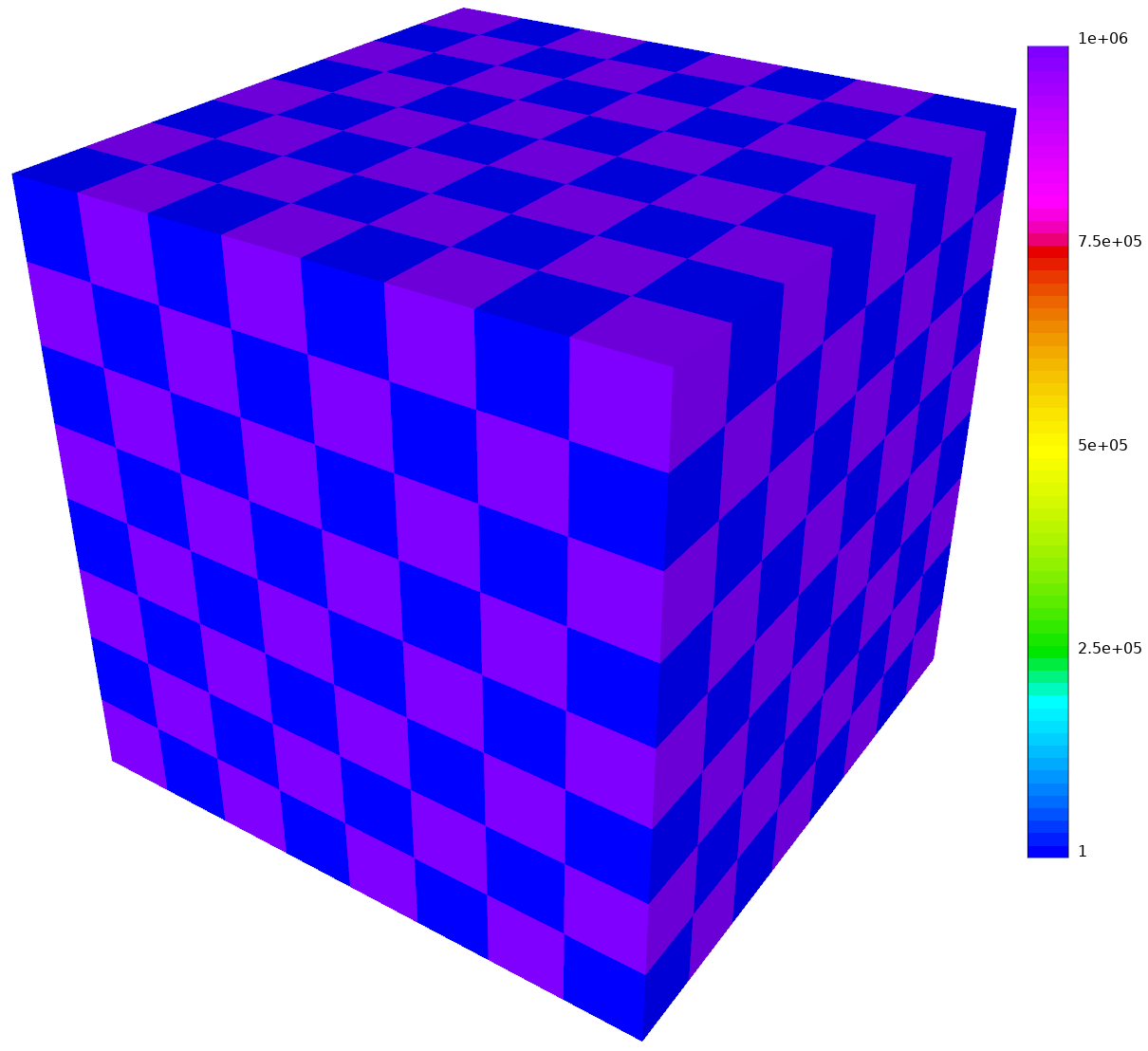}\label{fig:coefficient}}\quad
\subfloat[][Coefficient in 2D]{\includegraphics[width=0.485\textwidth]{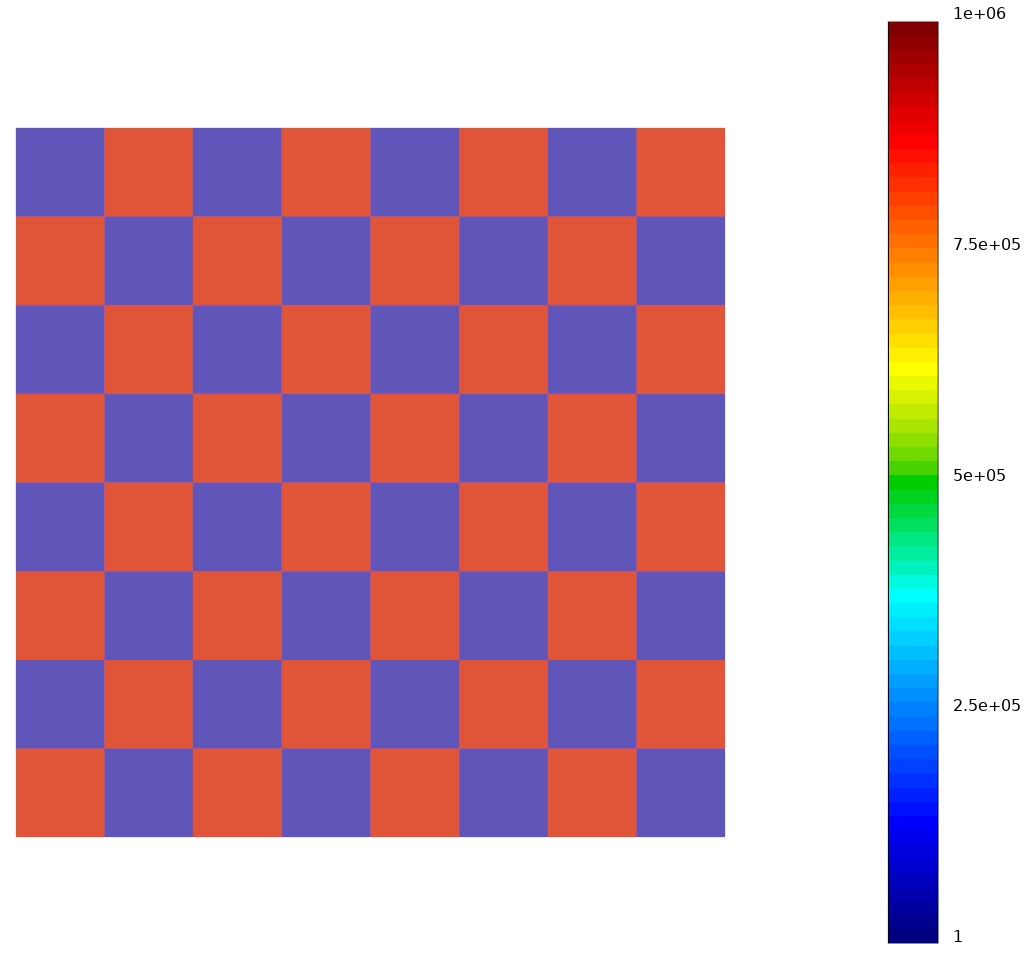}\label{fig:coefficient2D}}
\caption[]{The coefficient $\kappa$ in \eqref{eq:pde}.}
\end{figure}

A set of numerical results is shown. Two test cases are presented: on low and high order discretizations. First, the test setting is outlined.

\subsection{Setting}

The test problem is \eqref{eq:pde} with $f \equiv 1$, $\kappa$ representing the high contrast coefficients shown in \cref{fig:coefficient,fig:coefficient2D}, and $\delta = 1$ in \eqref{eq:localip}. Both $\cT^h$ and $\cT^H$, for the IP formulation, are regular and of the general type depicted in \cref{fig:ELEMENTS,fig:ELEMENTS2D}. Any further agglomeration, required by the procedure of \cref{sec:nonconfprecond}, is produced using METIS \cite{metis}. Note that the {\Elements} in \cref{fig:ELEMENTS,fig:ELEMENTS2D} get respectively refined as the mesh is refined, i.e., refining $\cT^h$ leads to a corresponding refinement of $\cT^H$. In contrast, the coefficients in \cref{fig:coefficient,fig:coefficient2D} remain fixed and their pattern is invariant with respect to mesh refinement.

A few measures of operator complexity (OC), representing relative sparsity in the obtained space hierarchies, are reported. Namely, the OC of the IP reformulation:
\[
\mathrm{OC}_\mathrm{IP} = 1 + \NNZ(\cA^{H_0})/\NNZ(A);
\]
the OC of the ``auxiliary'' multigrid hierarchy (e.g., generated by the method in \cref{sec:nonconfprecond}), relative to the IP matrix:
\[
\mathrm{OC}_{\mathrm{aux}} = 1 + \sum_{l=1}^{n_\ell} \NNZ(\cA^{H_l}) / \NNZ(\cA^{H_0});
\]
and the total OC, relative to the matrix $A$ in \eqref{eq:H1linsys}:
\[
\mathrm{OC}_{\mathrm{orig}} = 1 + \sum_{l=0}^{n_\ell} \NNZ(\cA^{H_l}) / \NNZ(A) = 1 + \mathrm{OC}_{\mathrm{aux}} \times (\mathrm{OC}_\mathrm{IP} - 1).
\]
Here, NNZ denotes the number of nonzero entries in the sparsity pattern of a matrix, $n_\ell$ is the number of levels (excluding the fine one) in the ``auxiliary'' multigrid hierarchy (e.g., generated by the method in \cref{sec:nonconfprecond}), and $\cA^{H_0} = \set{\cA,\cS,\cA^H,\text{ or }\cS^H}$ (depending on the case) is the matrix of the IP reformulation, involving or not coarsening and/or static condensation. The matrices $\cA^{H_l}$, for $l \ge 1$, represent the coarse versions of $\cA^{H_0}$ in the solver hierarchy built by the method in \cref{sec:nonconfprecond}.

Recall that dofs are associated with $\cU^h$ and the matrix $A$ in \eqref{eq:H1linsys}, whereas adofs, edofs, bdofs and {\Adofs}, {\Edofs}, {\Bdofs} are related to $\cE^h\times\cF^h$ and $\cE^H\times\cF^H$ and the matrices $\cA$, $\cS$ or $\cA^H$, $\cS^H$.

In all cases, the preconditioned conjugate gradient (PCG) method and the respective multiplicative auxiliary space preconditioner are applied for solving the linear system \eqref{eq:H1linsys} and the number of iterations, $n_\mathrm{it}$, is reported. The relative tolerance is $10^{-8}$ and the measure in the stopping criterion is $\br^T B^\mone \br$, for a current residual $\br$ and the respective preconditioner $B^\mone$ in the PCG method. The smoother in \cref{sec:smoother} is employed. Moreover,  coarsening (\cref{ssec:ipcoarseform}) is utilized in all cases and the particular preconditoner is $B^H_\mathrm{mult}$ in \eqref{eq:BH}. Thus, the problem is reduced to the choice of $(\cB^H)^\mone$ as an approximate inverse of $\cA^H$. Particularly, when static condensation (\cref{ssec:ipstatcond}) is employed, i.e., $(\cB^H)^\mone = (\cB_\mathrm{sc}^H)^\mone$ (a coarse version of \eqref{eq:Bsc}), the problem is reduced to the choice of $(S^H)^\mone$ as an approximate inverse of the respective $\cS^H$.

\subsection{Low order discretization}

\begin{table}
\centering
\small
\captionsetup[subfloat]{width=0.4\textwidth}
\subfloat[][$(\cB^H)^\mone = (\cA^H)^\mone$]
{\begin{tabular}{ | l | r | r | l | r | }
\hline
Refs & \# dofs & \# {\Adofs} & $\mathrm{OC}_\mathrm{IP}$ & $n_\mathrm{it}$ \\
\hline
0 & 4913 & 2240 & 1.304 & 2 \\
\hline
1 & 35937 & 17152 & 1.296 & 4 \\
\hline
2 & 274625 & 134144 & 1.291 & 6 \\
\hline
3 & 2146689 & 1060864 & 1.288 & 9 \\
\hline
4 & 16974593 & 8437760 & 1.287 & 12 \\
\hline
5 & 135005697 & 67305472 & 1.286 & 15 \\
\hline
\end{tabular}\label{tbl:loworderexact}}
\quad
\subfloat[][One AMGe (\cref{sec:nonconfprecond}) V-cycle as $(\cB^H)^\mone$]
{\begin{tabular}{ | l | r | l | l | r | }
\hline
Refs & $n_\ell+1$ & $\mathrm{OC}_\mathrm{aux}$ & $\mathrm{OC}_{\mathrm{orig}}$ & $n_\mathrm{it}$ \\
\hline
0 & 4 & 1.542 & 1.468 & 2 \\
\hline
1 & 5 & 1.711 & 1.506 & 4 \\
\hline
2 & 7 & 1.762 & 1.513 & 6 \\
\hline
3 & 8 & 1.753 & 1.506 & 11 \\
\hline
4 & 9 & 1.728 & 1.496 & 21 \\
\hline
5 & 9 & 1.708 & 1.489 & 41 \\
\hline
\end{tabular}\label{tbl:loworder}}
\caption[]{Low order test results with coarsening and without static condensation. That is, requiring an (approximate) inverse, $(\cB^H)^\mone$, of the respective $\cA^H$.}
\end{table}

\begin{table}
\centering
\small
\captionsetup[subfloat]{width=0.4\textwidth}
\subfloat[][$(\cB_\mathrm{sc}^H)^\mone$ with $(S^H)^\mone = (\cS^H)^\mone$]
{\begin{tabular}{ | l | r | r | l | r | }
\hline
Refs & \# dofs & \# {\Bdofs} & $\mathrm{OC}_\mathrm{IP}$ & $n_\mathrm{it}$ \\
\hline
0 & 4913 & 1728 & 1.647 & 2 \\
\hline
1 & 35937 & 13056 & 1.619 & 4 \\
\hline
2 & 274625 & 101376 & 1.604 & 6 \\
\hline
3 & 2146689 & 798720 & 1.597 & 9 \\
\hline
4 & 16974593 & 6340608 & 1.593 & 12 \\
\hline
5 & 135005697 & 50528256 & 1.591 & 15 \\
\hline
\end{tabular}\label{tbl:shurloworderexact}}
\quad
\subfloat[][$(\cB_\mathrm{sc}^H)^\mone$ with fixed three PCG iterations, preconditioned by an AMGe (\cref{sec:nonconfprecond}) V-cycle, as $(S^H)^\mone$]
{\begin{tabular}{ | l | r | l | l | r | }
\hline
Refs & $n_\ell+1$ & $\mathrm{OC}_\mathrm{aux}$ & $\mathrm{OC}_{\mathrm{orig}}$ & $n_\mathrm{it}$ \\
\hline
0 & 4 & 1.337 & 1.864 & 2 \\
\hline
1 & 5 & 1.571 & 1.973 & 4 \\
\hline
2 & 7 & 1.604 & 2.046 & 6 \\
\hline
3 & 8 & 1.814 & 2.083 & 9 \\
\hline
4 & 9 & 1.852 & 2.098 & 16 \\
\hline
5 & 9 & 1.871 & 2.106 & 28 \\
\hline
\end{tabular}\label{tbl:shurlowordersaamge}}\\
\subfloat[][$(\cB_\mathrm{sc}^H)^\mone$ with fixed one PCG iteration, preconditioned by a BoomerAMG V-cycle, as $(S^H)^\mone$]
{\begin{tabular}{ | l | r | l | l | r | }
\hline
Refs & $n_\ell+1$ & $\mathrm{OC}_\mathrm{aux}$ & $\mathrm{OC}_{\mathrm{orig}}$ & $n_\mathrm{it}$ \\
\hline
0 & 4 & 1.061 & 1.686 & 2 \\
\hline
1 & 5 & 1.090 & 1.675 & 4 \\
\hline
2 & 5 & 1.117 & 1.675 & 6 \\
\hline
3 & 6 & 1.153 & 1.688 & 9 \\
\hline
4 & 7 & 1.166 & 1.692 & 12 \\
\hline
5 & 9 & 1.185 & 1.701 & 17 \\
\hline
\end{tabular}\label{tbl:shurloworderboomeramg}}
\caption[]{Low order test results with coarsening and with static condensation. That is, a coarse version, $(\cB_\mathrm{sc}^H)^\mone$, of \eqref{eq:Bsc} is used, requiring an (approximate) inverse, $(S^H)^\mone$, of the respective $\cS^H$. Here, in two of the cases, $(S^H)^\mone$ is obtained invoking a fixed number of conjugate gradient iterations, preconditioned by multigrid V-cycles.}
\end{table}

Here, a 3D mesh of the type shown in \cref{fig:ELEMENTS} is sequentially refined and spaces, $\cU^h$, of piecewise linear finite elements are obtained. Coarse auxiliary spaces of \cref{ssec:ipcoarseform} are built with a single eigenvector per {\Element}, corresponding to the smallest eigenvalue. Static condensation (\cref{ssec:ipstatcond}) is very cheap in this case, involving the elimination of only one {\Edof} per {\Element}. The smoother in \cref{sec:smoother} is used throughout with $\nu = 4$. Also, whenever an AMGe hierarchy (\cref{sec:nonconfprecond}) is constructed, a single (smallest) eigenvector is taken from all local eigenvalue problems in \cref{sec:nonconfprecond} on all levels. 

First, no static condensation is employed and the IP problem is inverted (almost) exactly, resulting in $(\cB^H)^\mone = (\cA^H)^\mone$ in the auxiliary space preconditioner in \eqref{eq:BH}. Results are shown in \cref{tbl:loworderexact}. Then, $(\cB^H)^\mone$ is implemented via a single V-cycle of the solver in \cref{sec:nonconfprecond}; see \cref{tbl:loworder}. The number of iterations does not scale very well, in the particular case shown in \cref{tbl:loworder}, while trying to preserve reasonable cost per cycle of the method in \cref{sec:nonconfprecond}. This can be partially remedied by using static condensation (\cref{ssec:ipstatcond}) and wrapping the AMGe V-cycle in a few PCG iterations, as in \cref{tbl:shurlowordersaamge}, but the cost of each application of the action of $(S^H)^\mone$ is increased. However, invoking BoomerAMG (a well-known AMG implementation and a part of the HYPRE library\cite{hypre}) wrapped in a single PCG iteration provides a much more efficient alternative in this case. Indeed, observe that the number of iterations in \cref{tbl:shurloworderboomeramg} is almost the same as in \cref{tbl:shurloworderexact}, which involves the (almost) exact inversion of the respective Schur complement.

The AMGe method (\cref{sec:nonconfprecond}) demonstrates much better behavior in the high order tests below.

\subsection{High order discretization}
\label{ssec:highorder}

\begin{table}
\centering
\small
\captionsetup[subfloat]{width=0.4\textwidth}
\subfloat[][$(\cB_\mathrm{sc}^H)^\mone$ with $(S^H)^\mone = (\cS^H)^\mone$]
{\begin{tabular}{ | l | r | r | l | r | }
\hline
Order & \# dofs & \# {\Bdofs} & $\mathrm{OC}_\mathrm{IP}$ & $n_\mathrm{it}$ \\
\hline
1 & 289 & 136 & 1.738 & 2 \\
\hline
2 & 1089 & 144 & 1.117 & 6 \\
\hline
3 & 2401 & 756 & 1.816 & 5 \\
\hline
4 & 4225 & 980 & 1.557 & 4 \\
\hline
5 & 6561 & 1204 & 1.408 & 5 \\
\hline
6 & 9409 & 1428 & 1.313 & 7 \\
\hline
7 & 12769 & 1652 & 1.248 & 7 \\
\hline
8 & 16641 & 1876 & 1.201 & 4 \\
\hline
9 & 21025 & 2100 & 1.167 & 4 \\
\hline
10 & 25921 & 2324 & 1.141 & 3 \\
\hline
11 & 31329 & 2548 & 1.120 & 3 \\
\hline
12 & 37249 & 2772 & 1.104 & 4 \\
\hline
\end{tabular}}
\quad\quad
\subfloat[][$(\cB_\mathrm{sc}^H)^\mone$ with one AMGe (\cref{sec:nonconfprecond}) V-cycle as $(S^H)^\mone$]
{\begin{tabular}{ | l | r | l | l | r | }
\hline
Order & $n_\ell+1$ & $\mathrm{OC}_\mathrm{aux}$ & $\mathrm{OC}_{\mathrm{orig}}$ & $n_\mathrm{it}$ \\
\hline
1 & 2 & 1.685 & 2.243 & 3 \\
\hline
2 & 2 & 1.626 & 1.191 & 6 \\
\hline
3 & 3 & 1.254 & 2.023 & 6 \\
\hline
4 & 4 & 1.551 & 1.865 & 6 \\
\hline
5 & 5 & 1.758 & 1.717 & 8 \\
\hline
6 & 5 & 1.956 & 1.614 & 9 \\
\hline
7 & 6 & 2.185 & 1.541 & 9 \\
\hline
8 & 7 & 2.402 & 1.484 & 8 \\
\hline
9 & 7 & 2.618 & 1.437 & 9 \\
\hline
10 & 11 & 3.986 & 1.561 & 12 \\
\hline
11 & 12 & 4.318 & 1.520 & 14 \\
\hline
12 & 13 & 4.650 & 1.484 & 15 \\
\hline
\end{tabular}}
\caption[]{High order test results with coarsening and with static condensation. That is, a coarse version, $(\cB_\mathrm{sc}^H)^\mone$, of \eqref{eq:Bsc} is used, requiring an (approximate) inverse, $(S^H)^\mone$, of the respective $\cS^H$.}\label{tbl:highorder}
\end{table}

\begin{table}
\centering
\small
\captionsetup[subfloat]{width=0.4\textwidth}
\subfloat[][$(\cB^H)^\mone = (\cA^H)^\mone$]
{\begin{tabular}{ | l | r | r | l | r | }
\hline
Order & \# dofs & \# {\Adofs} & $\mathrm{OC}_\mathrm{IP}$ & $n_\mathrm{it}$ \\
\hline
1 & 289 & 200 & 1.573 & 3 \\
\hline
2 & 1089 & 208 & 1.090 & 6 \\
\hline
3 & 2401 & 984 & 1.430 & 5 \\
\hline
4 & 4225 & 1494 & 1.466 & 4 \\
\hline
5 & 6561 & 2130 & 1.492 & 5 \\
\hline
6 & 9409 & 3216 & 1.632 & 7 \\
\hline
7 & 12769 & 4982 & 1.891 & 7 \\
\hline
8 & 16641 & 7518 & 2.231 & 4 \\
\hline
9 & 21025 & 10978 & 2.666 & 4 \\
\hline
10 & 25921 & 15894 & 3.299 & 3 \\
\hline
\end{tabular}}
\quad\quad
\subfloat[][One AMGe (\cref{sec:nonconfprecond}) V-cycle as $(\cB^H)^\mone$]
{\begin{tabular}{ | l | r | l | l | r | }
\hline
Order & $n_\ell+1$ & $\mathrm{OC}_\mathrm{aux}$ & $\mathrm{OC}_{\mathrm{orig}}$ & $n_\mathrm{it}$ \\
\hline
1 & 2 & 1.791 & 2.026 & 3 \\
\hline
2 & 2 & 1.736 & 1.155 & 6 \\
\hline
3 & 3 & 1.231 & 1.529 & 10 \\
\hline
4 & 4 & 1.505 & 1.701 & 12 \\
\hline
5 & 5 & 1.779 & 1.875 & 16 \\
\hline
6 & 5 & 2.161 & 2.366 & 19 \\
\hline
7 & 6 & 2.562 & 3.282 & 25 \\
\hline
8 & 7 & 2.841 & 4.499 & 34 \\
\hline
9 & 7 & 3.161 & 6.267 & 46 \\
\hline
10 & 11 & 4.766 & 11.958 & 60 \\
\hline
\end{tabular}}
\caption[]{High order test results with coarsening and without static condensation. That is, requiring an (approximate) inverse, $(\cB^H)^\mone$, of the respective $\cA^H$.}\label{tbl:highordernoschur}
\end{table}

Intuitively, the IP reformulation is expected to exhibit more beneficial properties in the context of increasing the polynomial order of the finite element spaces, particularly in a combination with coarsening (\cref{ssec:ipcoarseform}) and static condensation (\cref{ssec:ipstatcond}). The available implementation is currently not matrix-free, therefore results in 2D are demonstrated using a fixed mesh of the type shown in \cref{fig:ELEMENTS2D} and increasing the polynomial order. Coarse auxiliary spaces described in \cref{ssec:ipcoarseform} are constructed with $\theta = 0.05$ and the static condensation in \cref{ssec:ipstatcond} is employed in \cref{tbl:highorder}, while \cref{tbl:highordernoschur} is produced utilizing the same parameters without static condensation. Notice that, for the condensed case, the {\Edofs} are eliminated, involving exclusively local coarse-scale work, leaving only {\Bdofs}. The smoother in \cref{sec:smoother} is used throughout with $\nu = 2$.

The method in \cref{sec:nonconfprecond} is constructed with $\theta_s = 0.05$. Since no mesh refinement is performed, no additional agglomeration is invoked. That is, the {\Elements} for the IP reformulation are maintained throughout the hierarchy and only the basis functions are reduced using the eigenvalue problems and imposing an additional requirement on the SVD filtering of linear dependencies on MISes. Namely, during the coarsening on each MIS, it is enforced that the SVD should remove at least 2 or 3 vectors relative to the current level being coarsened. This is reminiscent of $p$-multigrid but in a spectral AMGe setting.

Results, with static condensation, are shown in \cref{tbl:highorder}, demonstrating the especially good performance of the auxiliary space formulation -- the case of utilizing a coarse version, $(\cB_\mathrm{sc}^H)^\mone$, of \eqref{eq:Bsc} with $(S^H)^\mone = (\cS^H)^\mone$. The method of \cref{sec:nonconfprecond} also shows more reliable performance in the case of high order discretizations and static condensation. Maintaining all parameters the same and only disabling static condensation produces the results in \cref{tbl:highordernoschur}. Observe the differences in operator complexities as the order is increased and the declined performance of the AMGe solver of \cref{sec:nonconfprecond}. This illustrates the utility of static condensation, when high order discretizations are involved.

\section{Conclusions and future work}
\label{sec:conclusion}
In this paper, we have introduced a modification of the classical interior penalty (IP) discretization method by having an additional space of discontinuous functions associated with the interfaces between the subdomains (or elements) used to couple  the individual local bilinear forms. This allows for element-by-element (or subdomain-by-subdomain) assembly. The resulting modified IP bilinear form and a reduced (Schur complement) form of it, obtained by static condensation, are then utilized in the construction of a number of auxiliary space preconditioners. These are analyzed and their proven mesh-independence spectral equivalenvce properties are illustrated by a set of numerical tests on 2D and 3D second order scalar elliptic equations, including the case of high order finite element discretizations. The element-by-element assembly property of the modified IP bilinear form was also beneficial in the construction of element-based algebraic multigrid (AMGe) that is used to replace the exact inverses in the auxiliary space preconditioners. A possible feasible extension of this work is to utilize the studied  auxiliary space preconditioners with AMGe coarse solves in the setting of matrix-free solvers for high order finite element discretizations in combination with the polynomial smoothers presented in \cref{sec:smoother}. Moreover, applying the approach of this paper to formulations that are conforming in $H(\div)$ or $H(\curl)$ (see \cref{rem:general}) is an interesting future study.

\bibliographystyle{plain}
\bibliography{references.bib}

\end{document}